\documentclass[12pt]{article}
\usepackage{epsfig,psfrag,amsmath,amssymb,latexsym}
\usepackage{amscd }
\usepackage{color}
\usepackage{amsfonts}
\usepackage{graphicx}
\usepackage{wasysym}
\usepackage{cancel}
\newtheorem{fact}{Fact}[section]

\newtheorem{theorem}[fact]{Theorem}

\newtheorem{corollary}[fact]{Corollary}

\newtheorem{example}[fact]{Example}

\newtheorem{lemma}[fact]{Lemma}

\newenvironment{proof}[1][Proof]{\textbf{#1.} }{\ \rule{0.5em}{0.5em}}

\numberwithin{equation}{section}

\newcommand{\diag}{\operatorname{diag}}

\newcommand{\law}{\operatorname{law}}
\newcommand{\sk}[1]{\left\langle #1 \right\rangle}

\newcommand{\I}{\mathcal{I}}
\newcommand{\B}{\mathcal{B}}
\newcommand{\E}{\mathbb{E}}

\newcommand{\cL}{\mathcal{L}}
\newcommand{\s}{\mathcal{S}}

\def\R{{\mathbb R}}  
\def\N{{\mathbb N}}  
\def\Z{{\mathbb Z}}  
\def\C{{\mathbb C}}  
\def\F{{\mathcal{F}}}
\def\T{{\mathcal{T}}}
\def\G{{\mathcal{G}}}  

\begin{document}
\thispagestyle{empty}

\begin{center}
{\LARGE A supersymmetric approach to martingales
related to the vertex-reinforced jump process}\\[3mm]
{\large Margherita Disertori\footnote{Institute for Applied Mathematics
\& Hausdorff Center for Mathematics, 
University of Bonn,
Endenicher Allee 60,
D-53115 Bonn, Germany.
E-mail: disertori@iam.uni-bonn.de}
\hspace{1cm} 
Franz Merkl \footnote{Mathematical Institute, University of Munich,
Theresienstr.\ 39,
D-80333 Munich,
Germany.
E-mail: merkl@math.lmu.de
}
\hspace{1cm} 
Silke W.W.\ Rolles\footnote{Zentrum Mathematik, Bereich M5,
Technische Universit{\"{a}}t M{\"{u}}nchen,
D-85747 Garching bei M{\"{u}}nchen,
Germany.
E-mail: srolles@ma.tum.de}
\\[3mm]
{\small \today}}\\[3mm]
\end{center}

\begin{abstract}
Sabot and Zeng have discovered two martingales, one of which played a key role
in their investigation of the vertex-reinforced jump process.
Starting from the related supersymmetric hyperbolic sigma model, we 
give an alternative derivation of these two martingales. 
They turn out to be the first two instances in an infinite hierarchy of 
martingales, derived from a generating function. \\[1mm]
{\footnotesize Key words: nonlinear sigma model, vertex-reinforced jump process, 
martingale; }
\end{abstract}


\section{Introduction}
In \cite{sabot-tarres-zeng15}, Sabot, Tarr\`es, and Zeng proved that the
vertex-reinforced jump process can be related to a certain   random Schr\"odinger operator. 
A convenient way to characterize the corresponding random environment $\beta$ 
is its Laplace transform, investigated in \cite{sabot-tarres-zeng15} using a 
matrix decomposition from linear algebra. 

Subsequently,  Sabot and Zeng have discovered, in  \cite{sabot-zeng15},
 that a certain field $\psi^{(n)}$ associated to the random Schr\"odinger operator, 
on increasing finite pieces (with wired boundary conditions) of an 
infinite graph exhibits a martingale property. 
This turns out to be the crucial ingredient to prove, among other things, a 
characterization of the recurrence and transience behavior of the vertex-reinforced
jump process  on an arbitrary locally finite graph and, in a certain 
parameter regime, a functional central limit theorem for this process 
on $\Z^d$ with $d\ge 3$. Ergodicity with respect to spatial translations of 
the limit of the mentioned martingale is also one of the key ingredients 
for Sabot and Zeng's proof of recurrence of linearly edge-reinforced random 
walk on $\Z^2$ with arbitrary constant initial weights.  

Sabot and Zeng have also described the (discrete) quadratic variation of the 
mentioned martingale in terms of a second martingale
involving the Green's function of the random Schr\"odinger operator. 

In the present paper, we show that these two martingales are the first two 
instances of an infinite hierarchy of 
martingales, described in Corollary \ref{cor:further-mgs} below. The infinite 
hierarchy is obtained by expanding a martingale consisting of generating 
functions; cf.\ Theorem \ref{th:generating-exponential-martingale}. 

Our starting point is the supersymmetric hyperbolic sigma model $H^{2|2}$,
invented by Zirnbauer \cite{zirnbauer-91} and investigated by Disertori, Spencer, and 
Zirnbauer in \cite{disertori-spencer-zirnbauer2010}.
In \cite{sabot-tarres2012}, Sabot and Tarr\`es showed that this model is related
to the mixing measures for both vertex-reinforced jump process and edge-reinforced random walk.
Key ingredients in our analysis are the symmetries of $H^{2|2}$ and a local scaling transformation.

\paragraph{Overview of this article.}
In Section \ref{se:def-main-results}, Zirnbauer's $H^{2|2}$ model is 
defined formally and the main results are stated.

In Section \ref{subse:local-scaling-transf} we introduce the mentioned local
scaling transformation of the random field $(e^u,s)$, described by $H^{2|2}$.
In Theorem \ref{th:image-measure} we  describe the Radon-Nikodym  
derivative of the distribution of the transformed field with respect to the 
original random field.
It allows us also to give a short 
alternative proof of the Laplace transform of $\beta$ from \cite{sabot-tarres-zeng15};
cf.\ Corollary \ref{co:laplace-trafo-beta} below. 
Taking  $H^{2|2}$ as a starting point, the measurability argument
required to show the various martingale properties is a little easier 
than in the random Schr\"odinger operator approach. This is why we include 
the argument in Section \ref{subse:measurability}. 

In Section \ref{se:first-martingal}, using Theorem \ref{th:image-measure}
and the fact $\E[e^{u_k}]=1$ known from \cite{disertori-spencer-zirnbauer2010}, 
we give a short alternative proof for the first martingale from \cite{sabot-zeng15}; 
cf.\ Theorem \ref{th:first-mg} below. 
In addition to the local scaling transformation, our proofs in Sections 
\ref{se:first-martingal} and \ref{se:further-martingales} of the martingale
properties use a Kolmogorov consistency discovered by Sabot and Zeng 
\cite{sabot-zeng15} for the random environment $\beta$. 

In Section \ref{se:susy}, we first review the symmetries of $H^{2|2}$ that we need
for our proof. These include ordinary euclidean rotations and a $Q$-supersymmetry 
introduced in \cite{disertori-spencer-zirnbauer2010}. 
Using these (super-)symmetries, we derive Ward identities for 
certain harmonic functions; see Lemma \ref{le:rot-sym}.
The proof of this lemma is based on two main ingredients.
First, the mean value theorem for 
harmonic functions localizes the average over a circle at its center. 
Second, a technique from \cite{disertori-spencer-zirnbauer2010} localizes 
the expectation of $Q$-supersymmetric functions at the zero field configuration. 

In Section \ref{se:further-martingales}, a combination of these Ward identities 
with the local scaling transformation from 
Section \ref{subse:local-scaling-transf}
yields a generating martingale. An infinite sequence of martingales
is then produced by Taylor expansion. 

Finally, in Section \ref{se:letac}, we use Theorem \ref{thm:expectation-G}, which 
is also a basic ingredient for the generating martingale, to prove a
formula discovered by Letac (unpublished).

\section{Definitions and main results}
\label{se:def-main-results}

\subsection{Finite graph}\label{subse:finite-graph}
Let $\tilde\G=(\tilde V,\tilde E)$ be a finite, undirected, connected graph with vertex set 
$\tilde V$ and edge set $\tilde E$. The graph is assumed to have no self-loops. 
We fix a reference vertex $\delta\in \tilde V$ and abbreviate $V:=\tilde V\setminus\{\delta\}$. 
We assign positive weights $W_{ij}=W_{ji}>0$ to every undirected edge 
$(i\sim j)\in \tilde E$ and set $W_{ij}=0$ for $i\not\sim j$. In particular, $W_{ii}=0$. 
Let 
\begin{align}
U_V:= & 
\{ u=(u_i)_{i\in \tilde V}\in\R^{\tilde V}:\; u_\delta=0\}, \\
\label{eq:def-Omega}
\Omega_V:= & U_V\times U_V= 
\{ (u=(u_i)_{i\in \tilde V},s=(s_i)_{i\in \tilde V})\in\R^{\tilde V}\times\R^{\tilde V}:\; 
u_\delta=0,s_\delta=0\}.
\end{align}
For $u\in U_V$, we define the (negative) discrete Laplacian
$A^W(u)\in\R^{\tilde V\times \tilde V}$ associated to the weights $W_{ij}e^{u_i+u_j}$ by
\begin{align}
\label{eq:def-A}
A_{i,j}^W(u)=\left\{\begin{array}{ll}
-W_{ij}e^{u_i+u_j} & \text{for }i\neq j, \\
\sum_{k\in \tilde V} W_{ik}e^{u_i+u_k}  & \text{for }i=j.
\end{array}\right. 
\end{align}
Let $A^W_{VV}(u)$ denote the submatrix of $A^W(u)$ obtained by 
deleting the $\delta$-th row and column, and 
 $\T$  the set of spanning trees of $\tilde\G$.

The $H^{2|2}$ model on $\tilde{\G}$ is given by a probability measure
$\mu^W$ on $\Omega_V$. Following \cite{disertori-spencer-zirnbauer2010} 
and \cite{disertori-spencer2010},
it can be written in the two following equivalent ways:
\begin{align}
 & \mu^W(du\, ds) \nonumber\\
= & 
\prod_{(i\sim j)\in \tilde E} e^{-W_{ij}[\cosh(u_i-u_j)+\frac12(s_i-s_j)^2e^{u_i+u_j}-1]}
\sum_{T\in\T}\prod_{(i\sim j)\in T} W_{ij} e^{u_i+u_j} 
\prod_{i\in V} e^{-u_i}\, \tfrac{du_ids_i}{2\pi} 
\nonumber\\
=&\  e^{-\frac12 s^t A^{W}(u) s}\det A^W_{VV}(u) \prod_{(i\sim j)\in \tilde E} e^{-W_{ij}[\cosh(u_i-u_j)-1]}
\prod_{i\in V} e^{-u_i}\, \tfrac{du_ids_i}{2\pi}
\label{eq:def-mu-W}
\end{align}
with $du_{i}$ and $ds_{i}$ denoting the Lebesgue measure on $\R$. 
Recall that $s_{\delta }=0$; hence 
we need only the submatrix $A^W_{VV}$ to evaluate the quadratic form $s^t A^{W} s$.
Because the graph $\tilde\G$ is connected, this quadratic form with the 
constraint $s_\delta=0$ is positive definite. In particular, the matrix $A^W_{VV}$
is invertible.

We define the Green's function 
$\hat G=\hat G^V=\hat G^{V,W}:U_V\to\R^{\tilde V\times \tilde V}$ by 
\begin{align}
\label{eq:def-hat-G}
\hat G_{ij}(u)= \left\{\begin{array}{ll}
e^{u_i}(A^W_{VV}(u)^{-1})_{ij}e^{u_j} & \text{for }i,j\in V, \\
0 & \text{for }i=\delta\text{ or }j=\delta. 
\end{array}\right.
\end{align}
Note that this definition is equivalent to the representation of 
$\hat G$ given in formula (4.4) in
\cite{sabot-zeng15}.
Furthermore, we introduce the random vector 
$\beta^{V,W}(u)= (\beta_{i}^{V,W} (u))_{i\in V}$ by  
\begin{align}\label{eq:beta-definition}
\beta_i^{V,W}(u) =\frac12 \sum_{j\in \tilde V:j\sim i} W_{ij} e^{u_j-u_i}.
\end{align}
When there is no risk of confusion we use the notation $\beta$, $\beta^V$, or $\beta^W$ 
(according to which dependence we want to stress) instead of $\beta^{V,W}$. 

We denote by $\E_{\mu^W}$ the expectation with respect to $\mu^W$ and by 
$\sk{a,b}=\sum_{i\in I}a_ib_i$  
the euclidean scalar product, where $I=V$ or $I=\tilde V$ depending on the context. 
We will also need the following space:
\begin{align}
\label{eq:def-Lambda}
\Lambda_V:=\{ \lambda=(\lambda_i)_{i\in \tilde V}\in(-1,\infty)^{\tilde V}:\; 
\lambda_\delta=0\}.
\end{align}
For $\lambda\in\Lambda_V$, we denote by $\lambda_V$ its restriction to $V$. 
Real functions of $\lambda$, like $\sqrt{1+\lambda}$, are understood 
componentwise. We abbreviate $e^u_{\tilde V}=(e^{u_i})_{i\in\tilde V}$. 

The main result of this section is the following generalization of the Laplace transform of 
$\beta=(\beta_{i})_{i\in V}$.

\begin{theorem}
\label{thm:expectation-G}
For all $\theta\in(-\infty,0]^{\tilde V}$ and all $\lambda\in\Lambda_V$, one has 
\begin{align}
\label{eq:expectation-G-exp-beta}
\E_{\mu^W}\left[e^{\sk{\theta, e^u_{\tilde V}}-\frac12\sk{\theta,\hat G(u)\theta}}e^{-\sk{\lambda_V,\beta}} \right]
= \cL^W(\lambda) e^{\sk{\theta,\sqrt{1+\lambda}}},
\end{align}
where 
\begin{equation}\label{eq:Laplace}
\cL^W(\lambda)=  \prod_{(i\sim j)\in \tilde E} e^{W_{ij}(1-\sqrt{1+\lambda_i} \sqrt{1+\lambda_j} )}
\prod_{i\in V} \frac{1}{\sqrt{1+\lambda_i}}. 
\end{equation}
\end{theorem}
The proof is done in Section \ref{se:further-martingales}.
For $\theta=0$ equation \eqref{eq:expectation-G-exp-beta} gives indeed the Laplace 
transform $\cL^W(\lambda)$ of $\beta$. This special case appeared first in 
Proposition 1 of \cite{sabot-tarres-zeng15} in the context of
a random Schr\"odinger operator approach. 
The equivalence of this approach  to $H^{2|2}$ is shown
in Corollary~2  of \cite{sabot-tarres-zeng15}.
In particular the joint distribution of 
the $\beta_i$'s is a marginal of their $\nu^{W,1}$.
Using a local scaling transformation, we will give an alternative derivation 
of the Laplace transform $\cL^W(\lambda)$ in 
Corollary \ref{co:laplace-trafo-beta}.

For any vector $b=(b_i)_{i\in V}\in\R^{V}$, we define 
\begin{align}
\label{eq:def-H-beta}
(H_b)_{ij}= 2b_{i}\delta_{ij}- W_{ij},\quad i,j\in V.
\end{align}
Note that $(H_{\beta(u)})_{ij}=e^{-u_i}A^W_{ij}(u)e^{-u_j}$ for all $i,j\in V$, and hence
\begin{align}
\label{eq:relation-H-beta-hat-G}
H_{\beta(u)}^{-1}=\hat G^{V,W}_{VV}(u)
\end{align}
is the restriction of $\hat G^{V,W}(u)$, defined in \eqref{eq:def-hat-G}, to 
$V\times V$.  

The following result is a consequence of Theorem \ref{thm:expectation-G}.
\footnote{Xiaolin Zeng has told us that G\'erard Letac has proved formula 
\eqref{eq:letac-xiaolin} with an inductive approach using linear algebraic 
methods. Unfortunately, this proof is not published and we don't know it.
We were wondering whether Theorem \ref{thm:expectation-G} is related to 
Letac's formula. Xiaolin Zeng and Christophe Sabot have answered this question 
in the affirmative. Christophe Sabot (private communication) showed that 
Theorem \ref{thm:expectation-G}
can be derived from Letac's formula. Here, we go in the opposite direction 
and deduce Letac's formula from Theorem \ref{thm:expectation-G}. }

\begin{corollary}[Letac's formula]
\label{cor:letac}
For all $\phi,\theta\in(0,\infty)^V$, one has 
\begin{align}
\label{eq:letac-xiaolin}
\int_{\{b\in\R^V: H_b>0\}} 
\frac{e^{-\tfrac12\left(\sk{\phi,H_b\phi}+\sk{\theta,H_b^{-1}\theta}\right)}
}{\sqrt{\det H_b}}\, db
=\left(\frac{\pi}{2}\right)^{\tfrac{|V|}{2}}
\frac{e^{-\sk{\phi,\theta}}}{\prod_{i\in V}\phi_i},
\end{align}
where the notation $H_b>0$ means that $H_b$ is positive definite. 
\end{corollary}

To construct the martingale hierarchy, we will also need some  measurability result.
Recall that $\beta^V= (\beta_{i}^V)_{i\in V}$. 

\begin{lemma}[Measurability, Sabot-Tarr\`es-Zeng \cite{sabot-tarres-zeng15}]
\label{le:measurability}
\mbox{}\\
The random vector $(u_{i})_{i\in {\tilde V}}$ is measurable with respect to
 the $\sigma$-field $\sigma (\beta^V)$. 
In other words, there is a measurable function $f_V^W:\R^V\to\R^{\tilde V}$ such that 
\begin{align}
(u_i)_{i\in\tilde V}=f_V^W(\beta^V).
\end{align}
\end{lemma}
Given the equivalence of $H^{2|2}$ and a random Schr\"odinger description 
mentioned above, this lemma  follows from Theorem 3 in  \cite{sabot-tarres-zeng15}.
However, since our starting point is $H^{2|2}$ rather than random Schr\"odinger operators,
we include the proof in Section \ref{subse:measurability} below. 

\subsection{Infinite graph}\label{subse:infinite-graph}

Let $\G_\infty=(V_\infty,E_\infty)$ be an infinite locally finite connected 
undirected graph without direct loops. We approximate $\G_\infty$ by finite 
subgraphs $\G_n=(V_n,E_n)$ such that $V_n\uparrow V_\infty$ and 
$E_n=\{(i\sim j)\in E_\infty:i,j\in V_n\}$. Given $n$, we obtain a new finite
graph $\tilde\G_n=(\tilde V_n,\tilde E_n)$ from $\G_\infty$ by collapsing all 
vertices in $V_\infty\setminus V_n$ to a single vertex $\delta_n$. Thus, 
$\tilde V_n=V_n\cup\{\delta_n\}$ and 
\begin{align}
\tilde E_n= E_n\cup\{(i\sim\delta_n):i\in V_n\text{ and }\exists 
j\in V_\infty\setminus V_n\text{ such that }(i\sim j)\in E_\infty\}. 
\end{align}
In other words, $\tilde\G_n$ is obtained from $\G_n$ introducing wired boundary 
conditions.  
As in Section \ref{subse:finite-graph}, we assign positive weights $W_{ij}=W_{ji}>0$ 
to every undirected edge $(i\sim j)\in E_\infty$ and we set $W_{ij}=0$ for 
$i\not\sim j$. For $i,j\in\tilde V_n$, we define the weight 
$W^{(n)}_{ij}=W^{(n)}_{ji}$ as follows:
\begin{align}
& W^{(n)}_{ij}= W_{ij} \quad\text{ if } i\in V_n\text{ and }j\in V_n, \\
& W^{(n)}_{i\delta_n}=W^{(n)}_{\delta_ni}
=\sum_{j\in V_\infty\setminus V_n} W_{ij} 
\quad\text{ for }i\in V_n, \quad\text{and}\quad W^{(n)}_{\delta_n\delta_n}=0.
\end{align}
In particular, $W^{(n)}_{ij}>0$ if and only if $(i\sim j)\in\tilde E_n$. 

Let $\mu^W_n$ denote the $H^{2|2}$-measure defined in \eqref{eq:def-mu-W} 
for the graph $\tilde \G_n$ and edge weights $W^{(n)}_{ij}$. 
The following observation was made by Sabot and Zeng in \cite{sabot-zeng15}.
To make the presentation self-contained, we will repeat their argument in Section \ref{se:first-martingal}.

\begin{lemma}\emph{\textbf{(Kolmogorov consistency, 
Sabot-Zeng \cite{sabot-zeng15})}}
\label{le:consistency}
For $n\in\N$, the Laplace transform 
$\cL^W_n(\lambda)=\E_{\mu^W_n}[e^{-\sk{\lambda_{V_n},\beta^{V_n}}}]$ 
of $\beta^{V_n}=(\beta_i)_{i\in V_n}$  satisfies the consistency relation 
\begin{align}
\cL^W_n(\lambda)=\cL^W_{n+1}(\lambda),
\end{align}
for all $\lambda\in\Lambda_{V_{n+1}}$ with $\lambda_i=0$ for 
all $i\in \tilde V_{n+1}\setminus V_n$. 
In particular, the law of $\beta^{V_n}$ with respect to $\mu^W_n$ agrees with 
the law of $\beta^{V_{n+1}}|_{V_n}$ with respect to $\mu^W_{n+1}$.
\end{lemma}

Consequently, as worked out in \cite{sabot-zeng15}, Kolmogorov's extension 
theorem yields the existence of a coupling $({\pmb{\beta}}_i)_{i\in V_\infty}$ 
on a probability space $(\Omega_\infty,\F_\infty,\mu_\infty^W)$
such that for any $n\in\N$ the laws of the random vectors 
\begin{align}
{\pmb{\beta}}^{(n)}=({\pmb{\beta}}_i:\Omega_\infty\to\R)_{i\in V_n}
\end{align}
with respect to $\mu_\infty^W$ and $\beta^{V_n}:\Omega_{V_n}\to\R^{V_n}$ 
with respect to $\mu_n^W$ coincide; recall the definition \eqref{eq:def-Omega}
of $\Omega_{V_n}$. Following \cite{sabot-zeng15}, we define the $\sigma$-field
\begin{align}
\label{eq:def-filtration}
\F_n=\sigma({\pmb{\beta}}^{(n)})\subseteq\F_\infty.
\end{align}
Using the function $f_V^W$ from Lemma \ref{le:measurability}, we define 
\begin{align}
\label{eq:def-u-n}
u^{(n)}=(u_i^{(n)})_{i\in\tilde V_n}=f_{V_n}^W({\pmb{\beta}}^{(n)}). 
\end{align}
In particular, for all $n$, the law of $u^{(n)}$ with respect to $\mu_\infty^W$ 
coincides with the law of $u=(u_i)_{i\in\tilde V_n}$ with respect to $\mu_n^W$. 
We also define 
\begin{align}
\label{eq:extend-u-n}
u_i^{(n)}=u_{\delta_n}^{(n)}=0\quad\text{ for }i\in V_\infty\setminus V_n.
\end{align} 
In Section \ref{se:first-martingal}, we present an alternative short proof of the 
following first martingale from Proposition 9 in \cite{sabot-zeng15}.  

\begin{theorem}[Martingale $e^u$, Sabot and Zeng \cite{sabot-zeng15}]
\label{th:first-mg}
For any $k\in V_\infty$, the process $(e^{u_k^{(n)}})_{n\in\N}$ is a martingale with 
respect to the filtration $(\F_n)_{n\in\N}$:
\begin{equation}\label{eq:martingal-statement}
\E_{\mu_\infty^W}\left[e^{u_{k}^{(n+1)}}| \F_{n}\right]= e^{u_{k}^{(n)}},\qquad \forall k\in V_{\infty}.
\end{equation}
\end{theorem}
This martingale will now be generalized. Recall the definition 
\eqref{eq:def-hat-G} of the Green's function $\hat G$. 
We denote by $\hat G^{(n)}=\hat G^{V_n}(u^{(n)})$ the Green's function 
(on the graph $\tilde\G_n$) obtained by replacing $u$ by $u^{(n)}$.
Let $(-\infty,0]^{(V_\infty)}$ denote the set of all 
$\theta\in(-\infty,0]^{V_\infty}$ having only finitely many non-zero entries.
For these $\theta$ and $n\in\N$, we define 
$\theta^{(n)}\in (-\infty,0]^{\tilde{V}^{(n)}}$ by
\begin{align}
\label{def-delta-n}
\theta^{(n)}_i=\theta_i\quad\text{for }i\in V_n \qquad\text{and}\qquad
\theta^{(n)}_{\delta_n}=\sum_{j\in V_\infty\setminus V_n}\theta_j.
\end{align}

\begin{theorem}[Generating martingale]
\label{th:generating-exponential-martingale}
For all $\theta\in(-\infty,0]^{(V_\infty)}$,
\begin{align}
\label{eq:def-M-n-theta}
M^{(n)}(\theta)=e^{\sk{\theta^{(n)}, e^{u^{(n)}}}-\frac12\sk{\theta^{(n)},\hat G^{(n)}\theta^{(n)}}},
\quad n\in\N,
\end{align}
is a martingale with respect 
to the filtration $(\F_n)_{n\in\N}$ defined in \eqref{eq:def-filtration}.
\end{theorem}
The martingale $M^{(n)}(\theta)$ is the generating function for an infinite 
hierarchy of martingales. The first two martingales 
\eqref{eq:first-mg} and \eqref{eq:second-mg} in this hierarchy are the 
martingales discovered by Sabot and Zeng; see Proposition 9
in \cite{sabot-zeng15}.

In the following, we use the notation $\hat G^{(n)}_{kl}=\hat G^{(n)}_{lk}=0$ for 
$k\in V_\infty\setminus V_n$, $l\in V_\infty$.

\begin{corollary}[Hierarchy of martingales]
\label{cor:further-mgs}
For any $j,k,l\in V_\infty$, 
\begin{align}
\label{eq:first-mg}
& M_j^{(n)}=e^{u_j^{(n)}},\quad n\in\N,
\\
\label{eq:second-mg}
& M_{j,k}^{(n)}=e^{u_j^{(n)}+u_k^{(n)}}-\hat G^{(n)}_{jk},\quad n\in\N,
\qquad\text{ and }\qquad \\
\label{eq:third-mg}
& M_{j,k,l}^{(n)}=e^{u_j^{(n)}+u_k^{(n)}+u_l^{(n)}}
-e^{u_j^{(n)}}\hat G^{(n)}_{kl}
-e^{u_k^{(n)}}\hat G^{(n)}_{jl}
-e^{u_l^{(n)}}\hat G^{(n)}_{jk},
\quad n\in\N, 
\end{align} 
are martingales with respect to the filtration $(\F_n)_{n\in\N}$.
More generally, for any $m\in\N$ and any $i_1,\ldots,i_m\in V_\infty$,
\begin{align}
\label{eq:m-th-mg}
& M_{i_1,\ldots,i_m}^{(n)}=
\sum_{\substack{I\subseteq\{1,\ldots,m\}\\ |I| \text{ even}}}
\sum_{\mathcal I\in \mathcal{P}_2(I)}(-1)^{|I|/2}
\prod_{k\in \{1,\ldots,m\}\setminus I} e^{u^{(n)}_{i_k}}
\prod_{\{k,l\}\in \mathcal{I}} \hat G^{(n)}_{i_ki_l}
,\quad n\in\N
\end{align} 
are martingales with respect to the same filtration,
where $\mathcal{P}_2(I)$ denotes the set of all partitions of $I$
in sets of size $2$. 
\end{corollary}
Note that the case $I=\emptyset$ corresponds to the term 
$\prod_{k=1}^{m}e^{u^{(n)}_{i_k}}$ in the right-hand side of \eqref{eq:m-th-mg}.

\section{Some tools}\label{se:some-tools}
\subsection{Local scaling transformation}\label{subse:local-scaling-transf}

Fix $\lambda\in\Lambda_V$. We define the local shift 
\begin{align}
\label{eq:def-S-lambda}
S_\lambda:\Omega_V\to\Omega_V, \quad (\tilde u,s)\mapsto (u,s)
\quad\text{with }u_i=\tilde u_i+\log\sqrt{1+\lambda_i}\text{ for all }i\in \tilde V.
\end{align}
In particular, $S_\lambda$ leaves the $s$-variables unchanged and 
$\tilde u_\delta=u_\delta=0$. 
We also introduce the rescaled weights 
\begin{align}
W_{ij}^\lambda=W^{\lambda }_{ji}=\sqrt{1+\lambda_i}\sqrt{1+\lambda_j}W_{ij}.
\end{align}
The following theorem describes a key property of the local scaling transformation
$S_\lambda$. 

\begin{theorem}[Measure transformation]
\label{th:image-measure}
For all $\lambda\in\Lambda_V$, 
the image of $\mu^{W^\lambda}$ with respect to the transformation 
$S_\lambda$ is given by 
\begin{align}
\label{eq:claim-thm-image-measure}
S_\lambda \mu^{W^\lambda} (du\, ds)
= & 
\prod_{(i\sim j)\in \tilde E} e^{W_{ij}^\lambda-W_{ij}}
\cdot \prod_{j\in V} \sqrt{1+\lambda_j} \, e^{-\lambda_j\beta^W_j(u)} 
\,\mu^W(du\, ds)\\
= & 
\prod_{(i\sim j)\in \tilde E} e^{W_{ij}(\sqrt{1+\lambda_i} \sqrt{1+\lambda_j} -1)}
\cdot \prod_{j\in V} \sqrt{1+\lambda_j} \, e^{-\lambda_j\beta^W_j(u)} 
\,\mu^W(du\, ds) .
\nonumber 
\end{align}
\end{theorem}

\paragraph{Remark.}
Note that \eqref{eq:claim-thm-image-measure} gives the general formula for 
$S_\lambda \mu^W$: 
\begin{align}
S_\lambda \mu^{W} (du\, ds)= 
\prod_{(i\sim j)\in \tilde E} e^{W_{ij} - W_{ij}^{\lambda'}}
\cdot \prod_{j\in V} \frac{1}{\sqrt{1+\lambda_j'}}\, e^{\tfrac{\lambda_j'}{1+\lambda_j'} 
\beta^{W^{\lambda'}}_j(u)} 
\,\mu^{W^{\lambda'}}(du\, ds), 
\end{align}
where $\lambda'_j=-\lambda_j/(1+\lambda_j)$ so that 
$\sqrt{1+\lambda_j}\sqrt{1+\lambda_j'}=1$ and 
$W_{ij}^{\lambda'}=W_{ij}/(\sqrt{1+\lambda_i}\sqrt{1+\lambda_j})$.

\medskip\noindent\begin{proof}[Proof of Theorem \ref{th:image-measure}]
Using the definition \eqref{eq:def-mu-W} of $\mu^W$, we find 
\begin{align}
\label{eq:mu-W-lambda}
\mu^{W^\lambda}(d\tilde u\, ds) 
=&\  e^{-\frac12 s^t A^{W^\lambda}(\tilde u) s}\det A^{W^\lambda}_{VV}(\tilde u) 
\prod_{(i\sim j)\in \tilde E} e^{-W^\lambda_{ij}[\cosh(\tilde u_i-\tilde u_j)-1]}
\prod_{i\in V} e^{-\tilde u_i}\, \tfrac{d\tilde u_ids_i}{2\pi}. 
\end{align}
Fix $(u,s)\in \Omega_V$ and set 
$(\tilde u,s)=S_\lambda^{-1}(u,s)=((u_i-\log\sqrt{1+\lambda_i})_{i\in\tilde V},s)\in\Omega_V$ 
as in \eqref{eq:def-S-lambda}. 
From $W^\lambda_{ij} e^{\tilde u_i+\tilde u_j}= W_{ij} e^{u_i+u_j}$ for $i,j\in \tilde V$ 
one has $A^{W^\lambda}(\tilde u)=A^W(u)$ and hence 
\begin{align}
\label{eq:term-with-A}
e^{-\frac12 s^t A^{W^\lambda}(\tilde u) s}\det A^{W^\lambda}_{VV}(\tilde u) 
=e^{-\frac12 s^t A^W(u) s}\det A^W_{VV}(u). 
\end{align}
Again for $i,j\in\tilde V$, we calculate
\begin{align}
W^\lambda_{ij}\cosh(\tilde u_i-\tilde u_j)
& = \tfrac12 W_{ij}\sqrt{1+\lambda_i}\sqrt{1+\lambda_j} 
(e^{\tilde u_i-\tilde u_j}+e^{\tilde u_j-\tilde u_i}) \cr
&= \tfrac12 W_{ij}\sqrt{1+\lambda_i}\sqrt{1+\lambda_j} 
\left(\sqrt{\tfrac{1+\lambda_j}{1+\lambda_i}}e^{u_i-u_j}
+\sqrt{\tfrac{1+\lambda_i}{1+\lambda_j}}e^{ u_j- u_i}\right) \cr
&= \tfrac12 W_{ij} 
\left( (1+\lambda_j) e^{u_i-u_j} + (1+\lambda_i) e^{ u_j- u_i}\right) \cr
&= W_{ij} \cosh(u_i-u_j) + 
\tfrac12 W_{ij} \left(\lambda_j e^{u_i-u_j} + \lambda_i e^{ u_j- u_i}\right).
\end{align}
Summing this over all edges $i\sim j$ and using $W_{ij}=0$ for $i\not\sim j$, we get 
\begin{align}
\sum_{(i\sim j)\in\tilde E} W^\lambda_{ij}\cosh(\tilde u_i-\tilde u_j)
&= \sum_{(i\sim j)\in\tilde E} W_{ij} \cosh(u_i-u_j) + 
\tfrac12 \sum_{(i\sim j)\in\tilde E}
 W_{ij} \left(\lambda_j e^{u_i-u_j} + \lambda_i e^{ u_j- u_i}\right)\cr
&= \sum_{(i\sim j)\in\tilde E} W_{ij} \cosh(u_i-u_j) + 
\tfrac12 \sum_{j\in\tilde V} \lambda_j \sum_{i\in\tilde V} W_{ij} e^{u_i-u_j}  \nonumber\\
&= \sum_{(i\sim j)\in\tilde E} W_{ij} \cosh(u_i-u_j) + 
\sum_{j\in V} \lambda_j \beta_j^W(u),
\end{align}
where in the last line we used $\lambda_\delta=0$. Therefore, 
\begin{align}
\label{eq:last-but-one-step}
\prod_{(i\sim j)\in \tilde E} e^{-W^\lambda_{ij}[\cosh(\tilde u_i-\tilde u_j)-1]}
= \prod_{(i\sim j)\in \tilde E} e^{W_{ij}^\lambda-W_{ij}} 
\prod_{(i\sim j)\in \tilde E} e^{-W_{ij}[\cosh(u_i- u_j)-1]} 
\prod_{j\in V} e^{-\lambda_j \beta_j^W(u)}. 
\end{align}
Finally, 
\begin{align}
\label{eq:last-step}
\prod_{j\in V} e^{-\tilde u_j}
=\prod_{j\in \tilde V} e^{-\tilde u_j}
=\prod_{j\in \tilde V} \sqrt{1+\lambda_j} e^{-u_j},  
\end{align}
where we extended the product $\prod_{j\in V}$ to $\prod_{j\in \tilde V}$
using $\tilde u_\delta=0$. Substituting formulas 
\eqref{eq:term-with-A}, \eqref{eq:last-but-one-step}, and \eqref{eq:last-step}
into \eqref{eq:mu-W-lambda}, claim \eqref{eq:claim-thm-image-measure} follows. 
\end{proof}

\smallskip
The following corollary gives a short alternative
derivation of the Laplace transform of the random vector $(\beta_i)_{i\in V}$.
It is a special case of Theorem \ref{thm:expectation-G}
and also one of the ingredients for the proof of this theorem.

\begin{corollary}[Laplace transform of $\beta$, Sabot-Tarr\`es-Zeng \cite{sabot-tarres-zeng15}]
\label{co:laplace-trafo-beta}
\mbox{}\\
The func\-tion $\cL^W$, defined in formula \eqref{eq:Laplace}, 
is the Laplace transform of the 
random vector $\beta=(\beta_i)_{i\in V}$:
\begin{align}
\label{eq:laplace-trafo-beta}
\E_{\mu^W}\left[ e^{-\sk{\lambda_V,\beta}} \right]
= \cL^W(\lambda)=
\prod_{(i\sim j)\in \tilde E} e^{W_{ij}(1-\sqrt{1+\lambda_i} \sqrt{1+\lambda_j} )}
\prod_{i\in V} \frac{1}{\sqrt{1+\lambda_i}}
\end{align}
for all $\lambda\in\Lambda_V$.
\end{corollary}
\begin{proof}
Integrating both sides of \eqref{eq:claim-thm-image-measure} over $\Omega_V$, the claim 
follows from the fact that the image measure $S_\lambda \mu^{W^\lambda}$ 
is a probability measure on $\Omega_V$. 
\end{proof}

\smallskip
The following corollary contains the previous one as special case $g=1$:

\begin{corollary}
\label{cor:expectation-exp-fn}
For any random variable $g:\Omega_V\to\R$ and any $\lambda\in\Lambda_V$, one has 
\begin{align}
\label{eq:exp-fn-g}
\E_{\mu^W}\left[ g e^{-\sk{\lambda_V,\beta}} \right]
= \cL^W(\lambda) \E_{\mu^{W^\lambda}}[g\circ S_\lambda]
\end{align}
in the sense that the left-hand side exists if and only if the right-hand 
side exists. 
\end{corollary}
\begin{proof}
Using Corollary \ref{co:laplace-trafo-beta}, we rewrite claim 
\eqref{eq:claim-thm-image-measure} of Theorem \ref{th:image-measure} 
in the form 
\begin{align}
e^{-\sk{\lambda_V,\beta}} 
= \cL^W(\lambda) \frac{d (S_\lambda\mu^{W^\lambda})}{d\mu^W}.
\end{align}
This yields the claim as follows: 
\begin{align}
\E_{\mu^W}\left[ g e^{-\sk{\lambda_V,\beta}} \right]
= & \cL^W(\lambda) \E_{\mu^W}\left[ g \frac{d (S_\lambda\mu^{W^\lambda})}{d\mu^W} \right] \cr
= & \cL^W(\lambda) \E_{S_\lambda\mu^{W^\lambda}}[g]
= \cL^W(\lambda) \E_{\mu^{W^\lambda}}[g\circ S_\lambda]. 
\end{align}
\end{proof}

\begin{example}
\label{ex:expectation-exp-u}
Taking $g(u,s)=e^{u_k}$ for any $k\in\tilde V$, this corollary gives 
\begin{align}
\label{eq:ex-eu}
\E_{\mu^W}\left[ e^{u_k} e^{-\sk{\lambda_V,\beta(u)}} \right]
= \cL^W(\lambda)\sqrt{1+\lambda_k}.
\end{align}
\end{example}
Indeed, using
\begin{align}
g(S_\lambda(u,s))=e^{u_k+\log\sqrt{1+\lambda_k}}
=\sqrt{1+\lambda_k}e^{u_k},
\end{align}
formula \eqref{eq:exp-fn-g} reduces to formula \eqref{eq:ex-eu} as follows
\begin{align}
\E_{\mu^W}\left[ e^{u_k} e^{-\sk{\lambda_V,\beta(u)}} \right]
= \cL^W(\lambda)\sqrt{1+\lambda_k} \E_{\mu^{W^\lambda}}[e^{u_k}]
= \cL^W(\lambda)\sqrt{1+\lambda_k}.
\end{align}
The last equality follows from formula (B.3) in Appendix B of 
\cite{disertori-spencer-zirnbauer2010}, which shows $\E_{\mu^{W^\lambda}} [e^{u_k}]=1$. 
It is also a consequence of Corollary \ref{co:expectation-generating-fnc} below; 
cf.\ formula \eqref{eq:expectation-0}.

\subsection{Measurability}\label{subse:measurability}

\begin{proof}[Proof of Lemma \ref{le:measurability}]
Since $u_{\delta}=0$ we only need to study measurability of $(u_{i})_{i\in V}$.
Given $u\in\R^{\tilde V}$ with $u_\delta=0$, the definition 
\eqref{eq:beta-definition} of $\beta_{i}=\beta_i(u)$ can be reorganized
as
\begin{equation}\label{eq:measur1}
2\beta_{i} e^{u_{i}} - \sum_{j\in V} W_{ij} e^{u_{j}}= W_{i\delta}.
\end{equation}
Recall the definition of $H_b$ given in \eqref{eq:def-H-beta}. In particular, 
\begin{equation}\label{eq:measur2}
(H_{\beta})_{ij}= 2\beta_{i}\delta_{ij}- W_{ij}
=e^{-u_i}A^W_{ij}(u)e^{-u_j},\quad i,j\in V, 
\end{equation}
where the last equality follows from the definition \eqref{eq:beta-definition} of 
$\beta$ and \eqref{eq:def-A} of $A^W$. Since $A^W_{VV}$ is positive definite, 
the matrix $H_\beta$ is invertible. Using the notations $e^u_V=(e^{u_i})_{i\in V}$
and $W_{V\delta}=(W_{i\delta})_{i\in V}$, equation \eqref{eq:measur1}
above becomes $H_{\beta} e^{u}_V= W_{V\delta}$ or equivalently
\begin{align}
\label{eq:exp-u-in-terms-of-H-beta}
e^{u}_V= H_{\beta}^{-1} W_{V\delta}. 
\end{align}
Note that $H_\beta$ is a function
of $(\beta_i)_{i\in V}$ (and the fixed weights $W_{ij}$) only; 
hence it is $\sigma(\beta^V)$-measurable. 
Thus, $e^{u}_V$ is $\sigma(\beta^V)$-measurable. We define $f_V^W:\R^V\to\R^{\tilde V}$ 
on the range of $\beta^V$ by $\beta^V\mapsto u$ with $u_\delta=0$ and 
$u_V=\log (H_{\beta}^{-1} W_{V\delta})$, where the $\log$ is taken componentwise.
We then extend it in an arbitrary measurable way outside of the range of $\beta^V$. 
The claim follows. 
\end{proof}

\paragraph{Remark.} In our setup, starting with $H^{2|2}$, it is a priori clear
that $e^u_V$ has positive entries. As a consequence, $\log (H_{\beta}^{-1} W_{V\delta})$ 
is well-defined. In contrast to this, \cite{sabot-zeng15} starts with the distribution of 
the $\beta$'s. There, additional arguments are needed to insure that this $\log$ 
is indeed well-defined. 

\section{First martingale}
\label{se:first-martingal}

\begin{proof}[Proof of Lemma \ref{le:consistency} -- Kolmogorov consistency]
Using Corollary \ref{co:laplace-trafo-beta}, we can calculate both Laplace transforms:
\begin{align}
\label{eq:mg-expectation1}
\cL^W_{n+1}(\lambda)
= & \prod_{(i\sim j)\in\tilde E_{n+1}} e^{W_{ij}^{(n+1)}(1-\sqrt{1+\lambda_i} \sqrt{1+\lambda_j} )}
\prod_{i\in V_{n+1}} \frac{1}{\sqrt{1+\lambda_i}} ,\\
\label{eq:mg-expectation2}
\cL^W_n(\lambda)
= & \prod_{(i\sim j)\in\tilde E_n} e^{W_{ij}^{(n)}(1-\sqrt{1+\lambda_i} \sqrt{1+\lambda_j} )}
\prod_{i\in V_n} \frac{1}{\sqrt{1+\lambda_i}} . 
\end{align}
Since $\lambda_i=0$ for all $i\in V_{n+1}\setminus V_n$, the last product in 
\eqref{eq:mg-expectation1} agrees with the last product in 
\eqref{eq:mg-expectation2}. It remains to consider the product over edges. 
Let $(i\sim j)\in\tilde E_{n+1}$. We distinguish several cases.\\
{\it Case 1: $i\in V_n$ and $j\in V_n$.} 
Then $(i\sim j)\in\tilde E_n$ and $W_{ij}^{(n+1)}=W_{ij}=W_{ij}^{(n)}$. Thus 
the contribution of this edge is the same in \eqref{eq:mg-expectation1} and 
\eqref{eq:mg-expectation2}. \\
{\it Case 2: $i\in\tilde V_{n+1}\setminus V_n$ and 
$j\in\tilde V_{n+1}\setminus V_n$.} Then 
$W_{ij}^{(n+1)}(1-\sqrt{1+\lambda_i} \sqrt{1+\lambda_j} )=0$ because 
$\lambda_i=\lambda_j=0$. Furthermore, $(i\sim j)\not\in\tilde E_n$. Thus, 
$i\sim j$ does not contribute. \\
{\it Case 3: $i\in V_n$ and $j\in\tilde V_{n+1}\setminus V_n$.} For 
a fixed $i\in V_n$, we calculate 
\begin{align}
&  \sum_{\substack{j\in\tilde V_{n+1}\setminus V_n:\\ (i\sim j)\in\tilde E_{n+1}}}
W_{ij}^{(n+1)}(1-\sqrt{1+\lambda_i} \sqrt{1+\lambda_j} ) 
=  \Bigg[ W_{i\delta_{n+1}}^{(n+1)} + 
\sum_{j\in V_{n+1}\setminus V_n} W_{ij} \Bigg] ( 1-\sqrt{1+\lambda_i})
\cr
&\qquad\qquad 
= \Bigg[ \sum_{j\in V_\infty\setminus V_n} W_{ij} \Bigg] ( 1-\sqrt{1+\lambda_i}) 
= W_{i\delta_n}^{(n)}(1-\sqrt{1+\lambda_i} ). 
\end{align}
This is the contribution of the edge $(i\sim\delta_n)\in\tilde E_n$ to 
\eqref{eq:mg-expectation2}. 

Thus the products in \eqref{eq:mg-expectation1} and 
\eqref{eq:mg-expectation2} agree and the claim is proved.
\end{proof}

\smallskip\noindent 
\begin{proof}[Proof of Theorem \ref{th:first-mg} -- Martingale $e^u$]
Given $n\in\N$, it suffices to consider $k\in V_{n+1}$, since otherwise 
$u_{k}^{(n+1)}=u_{k}^{(n)}=0$ 
and \eqref{eq:martingal-statement} is trivially satisfied. Note that 
$u_k^{(n)}=u_{\delta_n}^{(n)}=0$ for $k\in V_{n+1}\setminus V_n$. 
By its definition \eqref{eq:def-u-n}, $u_k^{(n)}$ is $\F_n$-measurable.
It remains to prove 
\begin{align}
\label{eq:teilbeh-mg}
\E_{\mu_{\infty}^W}\left[e^{u_{k}^{(n+1)}} g ({\pmb{\beta}}^{(n)})\right]=
\E_{\mu_{\infty}^W} \left[e^{u_{k}^{(n)}} g ({\pmb{\beta}}^{(n)})\right]
\end{align}
for any measurable function $g:\R^{V_n}\to[0,\infty)$. 
For any given $c\in\R$, the uniqueness theorem for Laplace transforms allows us to 
restrict the claim to test functions 
$g({\pmb{\beta}}^{(n)})=\prod_{j\in V_n} e^{-\lambda_j{\pmb{\beta}}_j(u)}$ with 
$\lambda_j>c$ for all $j\in V_n$
\begin{align}
\label{eq:equivalent-claim1}
\E_{\mu_{\infty}^W}\left[e^{u_{k}^{(n+1)}} \prod_{j\in V_n} e^{-\lambda_j{\pmb{\beta}}_j} \right]=
\E_{\mu_{\infty}^W} \left[e^{u_{k}^{(n)}} \prod_{j\in V_n} e^{-\lambda_j{\pmb{\beta}}_j} \right], 
\end{align}
as long as all these expectations are finite. 
As explained below formula \eqref{eq:def-u-n}, the law of 
$u^{(n)}$ with respect to $\mu_\infty^W$ 
coincides with the law of $u=(u_i)_{i\in\tilde V_n}$ with respect to $\mu_n^W$. 
In analogy to \eqref{eq:extend-u-n}, we define 
$u':\Omega_{V_n}\to\R^{V_{n+1}}$ by 
\begin{align}
u_{k}'=\left\{ \begin{array}{ll}
u_{k}, &\text{ if } k\in V_{n},\\
u_{\delta_n}=0, &\text{ if } k\in V_{n+1}\setminus V_{n}.
\end{array} \right.
\end{align}
Then, claim \eqref{eq:equivalent-claim1} is equivalent to 
\begin{align}
\label{eq:expectation-mg1}
\E_{\mu^W_{n+1}}\Big[ e^{u_k} \prod_{j\in V_n} e^{-\lambda_j\beta_j^{V_{n+1}}(u)} \Big]
=\E_{\mu^W_n}\Big[ e^{u_k'} \prod_{j\in V_n} e^{-\lambda_j\beta_j^{V_n}(u)} \Big].
\end{align}
For $c=-1$, Corollary \ref{cor:expectation-exp-fn} and Example 
\ref{ex:expectation-exp-u} imply that these expectations are 
finite; hence the same is true for the expectations in \eqref{eq:equivalent-claim1}.

Set $\lambda_i=0$ for all $i\in\tilde V_{n+1}\setminus V_n$. 
Using Example \ref{ex:expectation-exp-u} and Lemma \ref{le:consistency}, 
we obtain the claim 
\eqref{eq:expectation-mg1} in both cases, $k\in V_n$ or $k\in V_{n+1}\setminus V_n$, 
as follows: 
\begin{align}
\E_{\mu^W_{n+1}}\Big[ e^{u_k} \prod_{j\in V_n} e^{-\lambda_j\beta_j^{V_{n+1}}(u)} \Big]
= & \E_{\mu^W_{n+1}}\Big[ e^{u_k} \prod_{j\in V_{n+1}} e^{-\lambda_j\beta_j^{V_{n+1}}(u)} \Big]
=\cL^W_{n+1}(\lambda)\sqrt{1+\lambda_k}\cr
=& \cL^W_n(\lambda)\sqrt{1+\lambda_k}
=\E_{\mu^W_n}\Big[ e^{u_k'} \prod_{j\in V_n} e^{-\lambda_j\beta_j^{V_n}(u)} \Big]. 
\end{align}
\end{proof}

\section{Using (super-)symmetries of the model}
\label{se:susy}

Let $\tilde\G=(\tilde V,\tilde E)$ be a finite graph as described at the beginning of Section 
\ref{se:def-main-results}. 
Disertori, Spencer, and Zirnbauer \cite{disertori-spencer-zirnbauer2010}
use an alternative representation in terms of Grassmann variables of the 
$H^{2|2}$ measure $\mu^W$ defined in \eqref{eq:def-mu-W}. It has the advantage
of making the internal symmetries and supersymmetries of the model visible. 
Since we are using these symmetries in the remainder, we briefly review 
this alternative representation; cf.\ Section 2.2 of 
\cite{disertori-spencer-zirnbauer2010}. 
Let $\overline\psi_i,\psi_i$, $i\in V$, be independent 
Grassmann variables, and let $\overline\psi_\delta=0=\psi_\delta$. The measure
$\mu^W$ can be represented as follows
\begin{align}
\mu^W(du\, ds) 
= & \prod_{i\in V}e^{-u_i} \frac{du_ids_i}{2\pi}
\partial_{\overline{\psi}_i}\partial_{\psi_i} e^{\s}
\end{align}
with the action $\s=\s(u,s,\overline\psi,\psi)$ given by 
\begin{align}
\s=-\sum_{(i\sim j)\in \tilde E} W_{ij}[\cosh(u_i-u_j)-1 +[\tfrac12(s_i-s_j)^2+
(\overline{\psi}_i-\overline{\psi}_j)(\psi_i-\psi_j)]e^{u_i+u_j}].
\end{align}
Thus, $\mu^W$ is the marginal of the superintegration form 
\begin{align}
D\mu^W = & \prod_{i\in V}e^{-u_i} \frac{du_ids_i}{2\pi}
\partial_{\overline{\psi}_i}\partial_{\psi_i} \circ e^{\s}
\end{align}
obtained by integrating the Grassmann variables out. 
Note that since $\overline\psi$ and $\psi$ are nilpotent, 
$e^{\s}$ can be written as a polynomial in these variables whose coefficients
are integrable functions of $u$ and $s$. 

The internal 
(super-)symmetries of $D\mu^W$ are most easily seen in cartesian 
coordinates $x=(x_i)_{i\in\tilde V}$, $y=(y_i)_{i\in\tilde V}$, 
$z=(z_i)_{i\in\tilde V}$, $\xi=(\xi_i)_{i\in\tilde V}$, and 
$\eta=(\eta_i)_{i\in\tilde V}$ defined by
\begin{align}
\label{eq:change-of-coord1}
& x_i=\sinh u_i-\left(\frac12s_i^2+\overline\psi_i\psi_i\right)e^{u_i}, \quad
y_i=s_ie^{u_i}, \quad
\xi_i=e^{u_i}\overline\psi_i, \quad
\eta_i=e^{u_i}\psi_i, \\
\label{eq:change-of-coord2}
& z_i=\sqrt{1+x^2_i+y^2_i+2\xi_i\eta_i}=\cosh u_i+\left(\frac12s_i^2+\overline\psi_i\psi_i\right)e^{u_i}.
\end{align}
In particular, $x_\delta=y_\delta=\xi_\delta=\eta_\delta=0$ and $z_\delta=1$.  
As described in sections 2.1 and 2.2 of \cite{disertori-spencer-zirnbauer2010},
the image of $D\mu^W$ under this supertransformation is given by formulas
(2.5) and (2.6) of that paper: 
\begin{align}
D\mu^W= \Big( \prod_{i\in V}
\frac{dx_idy_i}{2\pi}\, \partial_{\xi_i}\partial_{\eta_i}\circ
\frac{1}{z_i} \Big) e^{\s}
\end{align}
with the transformed action $\s=\s(x,y,\xi,\eta)$ given by
\begin{align}
\s=-\sum_{i,j\in V}W_{ij}(z_iz_j-(1+x_ix_j+y_iy_j+\xi_i\eta_j-\eta_i\xi_j))
-\sum_{i\in V}W_{i\delta}(z_i-1).
\end{align}
Note that when taking the Taylor expansion of $e^{\s}\cdot\prod_{i\in V}\frac{1}{z_i}$ 
in the new Grassmann variables,  
the coefficients are now functions of $x$ and $y$ which decay exponentially fast 
as $\|(x,y)\|\to\infty$.  
For any superfunction $F(s,u,\psi,\overline{\psi})$, we define 
$\E_{D\mu^W}[F]:=\int D\mu^W\, F$, whenever the integral exists. This is 
equivalent to require that the highest order term for the Taylor 
expansion of $e^{\s}\prod_{i\in V}e^{-u_i}F$ in the Grassmann variables is
an integrable function of $u$ and $s$. 
In the following we will consider only functions with enough regularity 
such that integrability holds also in the new coordinate system $x,y,\xi,\eta$. 

\paragraph{Rotational symmetry.} 
Using this representation, it is obvious that $D\mu^W$ is invariant with 
respect to rotations in the $xy$-plane, 
$(x,y,\xi,\eta)\mapsto (x^\alpha,y^\alpha,\xi,\eta)$ with 
\begin{align}
\label{eq:rot}
x^\alpha=x\cos\alpha - y\sin \alpha, \quad
y^\alpha=x\sin\alpha+y\cos\alpha, \quad\text{for }\alpha\in\R.
\end{align}
In horospherical coordinates $u,s,\overline\psi,\psi$ this symmetry is not 
so easy to describe and somehow hidden. 

\paragraph{$Q$-supersymmetry.} 
In \cite{disertori-spencer-zirnbauer2010}, the invariance of the $H^{2|2}$-model
with respect to the supersymmetry operator 
\begin{align}
Q= \sum_{i\in V}
(x_i\partial_{\eta_i}-y_i\partial_{\xi_i}+\xi_i\partial_{x_i}+\eta_i\partial_{y_i})
\end{align}
played a key role. In particular, Proposition 2 in Appendix C of 
\cite{disertori-spencer-zirnbauer2010} states that for any 
smooth superfunction $F=F(u,s,\overline\psi,\psi)$ with $QF=0$ and 
$e^{\s}F$ being integrable, one has 
\begin{align}
\label{eq:integral-Q-symm}
\E_{D\mu^W}[ F]= e^{\s(o)}F(o)=F(o), 
\end{align}
where $o$ denotes the zero-field configuration $u=s=0$, $\overline\psi=\psi=0$.
In particular, the assumption $QF=0$ is satisfied for smooth superfunctions 
of the form $F=F(z)$ because of $Qz_i=0$ for all $i\in V$. 

These (super-)symmetries play the key role in the proof of the 
following lemma. 

\begin{lemma}[Ward identities]
\label{le:rot-sym}
Let $f:\R^2\to\C$ be a harmonic function and $\theta\in\R^{\tilde V}$. 
We assume 
that any coefficient of the Taylor expansion in $\xi,\eta$ of the superfunction 
$f(\sk{\theta, x+z},\sk{\theta, y})e^{\s}$ decays at least exponentially fast 
in $\|(x,y)\|$ at infinity. Then, the following identity holds 
\begin{align}
\E_{D\mu^W}[f(\sk{\theta, x+z},\sk{\theta, y})]=f(\sk{\theta,1},0),
\end{align}
where $\sk{\theta,1}$ stands for $\sum_{i\in \tilde V} \theta_i$.
\end{lemma}
Note that the extension of $f$ to a superfunction is used in the
expectation because $z$ defined in \eqref{eq:change-of-coord2} involves 
Grassmann variables. This extension is denoted by the same symbol $f$. 

\smallskip\noindent
\begin{proof}
By rotational symmetry of the model $H^{2|2}$ in the $xy$-plane,
using the notation \eqref{eq:rot}, we have 
\begin{align}
\E_{D\mu^W}[f(\sk{\theta, x+z},\sk{\theta, y})]=
\E_{D\mu^W}[f(\sk{\theta, x^\alpha +z},
\sk{\theta, y^\alpha})]
\end{align}
for any $\alpha\in\R$. Taking the average over $\alpha\in[0,2\pi]$
and using the mean value theorem for the harmonic function $f$ yields
\begin{align}
\E_{D\mu^W}[f(\sk{\theta, x+z},\sk{\theta, y})]
& =
\frac{1}{2\pi}\int_0^{2\pi}
\E_{D\mu^W}[f(\sk{\theta, x^\alpha+z},
\sk{\theta, y^\alpha})]\,d\alpha
\cr&
=
\E_{D\mu^W}\left[\frac{1}{2\pi}\int_0^{2\pi}f(\sk{\theta, x^\alpha+z},
\sk{\theta, y^\alpha})\,d\alpha\right]
\cr&=
\E_{D\mu^W}\left[f(\sk{\theta, z},0)\right].
\end{align}
Since $f(\sk{\theta, z},0)$ is a smooth superfunction of $z$, we have 
the supersymmetry 
\begin{align}
Qf(\sk{\theta, z},0)=0. 
\end{align}
The assumption on exponential decay in the lemma implies that we can 
apply Proposition~2 from Appendix C of 
\cite{disertori-spencer-zirnbauer2010}, cited in \eqref{eq:integral-Q-symm}, 
to the averaged superfunction $f(\sk{\theta, z},0)$. It yields 
\begin{align}
\E_{D\mu^W}\left[f(\sk{\theta, z},0)\right] =f(\sk{\theta, 1},0).
\end{align}
\end{proof}

\begin{corollary}[Ward identity for $\exp$]
\label{co:expectation-generating-fnc}
For all $\theta\in(-\infty,0]^{\tilde V}$, one has 
\begin{align}
\label{eq:expectation-generating-fnc}
\E_{\mu^W}[e^{\sk{\theta, e^u(1+is)}}]
&=
e^{\sk{\theta, 1}},
\end{align}
using the abbreviation $e^u(1+is)=(e^{u_j}(1+is_j))_{j\in\tilde V}$.
\end{corollary}
\begin{proof}
From  \eqref{eq:change-of-coord1} and \eqref{eq:change-of-coord2},
we know $x_j+z_j=e^{u_j}$ and $y_j=s_je^{u_j}$,
and hence
\begin{align}
\E_{\mu^W}[e^{\sk{\theta, e^u(1+is)}}]=\E_{D\mu^W}[e^{\sk{\theta, x+z+iy}}].
\end{align}
We apply now Lemma \ref{le:rot-sym} to the holomorphic (and hence harmonic) 
function $f=\exp:\R^2=\C\to\C$, $f(x,y)=e^{x+iy}$. 
Note that the assumption $\theta_i\le 0$ and the superexponentially 
fast decay of $e^{\s}$ imply that the exponential decay condition 
is satisfied. We obtain
\begin{align}
\E_{D\mu^W}[e^{\sk{\theta, x+z+iy}}]=e^{\sk{\theta,1}},
\end{align}
which proves the claim \eqref{eq:expectation-generating-fnc}.
\end{proof}

\paragraph{Remark.}
As a consequence of Corollary \ref{co:expectation-generating-fnc}, 
we obtain for all vertices $k,l,m\in \tilde V$, 
\begin{align}
\label{eq:expectation-0}
& \E_{\mu^W}[e^{u_k}]=1, \\
& \E_{\mu^W}[e^{u_k+u_l}(1-s_ks_l)]=1, 
\label{eq:expectation-1}\\
\label{eq:expectation-2}
& \E_{\mu^W}[e^{u_k+u_l+u_m}(1-s_ks_l-s_ks_m-s_ls_m)]=1.
\end{align}
More generally, for any $m\in\N$ and any $i_1,\ldots, i_m\in\tilde{V}$,
\begin{align}
\label{eq:expectation-m}
\E_{\mu^W}\left[e^{\sum_{j=1}^m u_{i_j}}
\sum_{\substack{I\subseteq \{1,\ldots,m\}:\\|I|\text{ even}}}(-1)^{|I|/2}
\prod_{k\in I}s_{i_k}\right]=1.
\end{align}
Indeed, given $m\in\N$ and $i_1,\ldots, i_m\in\tilde{V}_n$,
we take the left derivative $\partial_{\theta_{i_1}}\ldots \partial_{\theta_{i_m}}$
at $\theta=0$ of \eqref{eq:expectation-generating-fnc} to
get
\begin{align}
\E_{\mu^W}\left[\prod_{k=1}^me^{u_{i_k}}(1+is_{i_k})\right]=1.
\end{align}
Note that 
the hypothesis $\theta_i\le 0$ allows us to interchange expectations 
and partial derivatives.
Expanding the product and taking
the real part of this equation gives formula
\eqref{eq:expectation-m}. The cases $m=1,2,3$ of this formula 
may be written in the form 
\eqref{eq:expectation-0}, \eqref{eq:expectation-1},
and \eqref{eq:expectation-2}, respectively. 

Using $u_\delta=s_\delta=0$, note that the $(m+1)$-st instance of formula 
\eqref{eq:expectation-m} contains the $m$-th instance as 
special case $i_{m+1}=\delta$.

\section{A hierarchy of martingales}
\label{se:further-martingales}

For a finite graph $\tilde\G=(\tilde V,\tilde E)$ with $\delta\in\tilde V$, 
recall the definitions \eqref{eq:def-A} of the matrix $A^W$ 
and \eqref{eq:def-hat-G} of the Green's function $\hat G$. 
We remind that the Gaussian part in the measure $\mu^W$ defined in 
\eqref{eq:def-mu-W} can be rewritten as 
\begin{align}
\label{eq:Gaussian-with-A}
\prod_{(i\sim j)\in \tilde E} e^{-\frac12 W_{ij}(s_i-s_j)^2e^{u_i+u_j}}
= e^{-\frac12 s^t A^W(u) s}. 
\end{align}
Therefore, we have the following representations
of the Green's function as conditional expectation:  
\begin{align}
\label{eq:formula-hat-G-as-cond-expect}
\hat G_{ij}&=\E_{\mu^W}[s_is_je^{u_i+u_j}|u]
\quad\mu^W\text{-a.s.},\quad\text{for all }i,j\in \tilde V,\\
\label{eq:formula-Gaussian-hat-G-as-cond-Laplace}
e^{-\frac12\sk{\theta,\hat G\theta}}&=\E_{\mu^W}[e^{i\sk{\theta,se^u}}|u]
\quad\mu^W\text{-a.s.}, \quad \text{ for any }\theta\in\R^{\tilde V}.
\end{align}
Note that $s_is_je^{u_i+u_j}\in L^p(\Omega_V,\mu^W)$
implies $\hat G_{ij}\in L^p(\Omega_V,\mu^W)$ for all $p\in[1,\infty)$. 

To prove Theorem \ref{th:generating-exponential-martingale}, we need some 
preliminary results. 
Since the martingale $M_n(\theta)$ in that theorem involves the 
Green's function and we use the preceding representation as a 
conditional expectation, we need the following variant of 
Corollary \ref{cor:expectation-exp-fn}.

\begin{lemma}
\label{le:expectation-exp-cond-exp}
For any random variable $g:\Omega_V\to\R$ and any $\lambda\in\Lambda_V$,
one has 
\begin{align}
\label{eq:exp-fn-g-cond-exp}
\E_{\mu^W}\left[ \E_{\mu^W}[g|u]\, e^{-\sk{\lambda_V,\beta}} \right]
= \cL^W(\lambda) \E_{\mu^{W^\lambda}}[g\circ S_\lambda]
\end{align}
in the sense that the left-hand side exists if and only if the right-hand 
side exists. 
\end{lemma}
\begin{proof}
This follows immediately from Corollary \ref{cor:expectation-exp-fn}
because $\beta$ is a function of $u$, but not of $s$.
\end{proof}

>From this lemma we get immediately the proof of Theorem \ref{thm:expectation-G}.
\smallskip

\noindent
\begin{proof}[Proof of Theorem \ref{thm:expectation-G}]
Using the conditional Laplace transform 
\eqref{eq:formula-Gaussian-hat-G-as-cond-Laplace}
and the fact that $\beta$ is a function of $u$ only, we can rewrite
the claim with the function $g_\theta(u,s)=e^{\sk{\theta, e^u(1+is)}}$ as follows:
\begin{align}
\E_{\mu^W}\left[ \E_{\mu^W}[g_\theta|u]\, e^{-\sk{\lambda_V,\beta}} \right]
= \cL^W(\lambda) e^{\sk{\theta,\sqrt{1+\lambda}}}.
\end{align}
We apply Lemma \ref{le:expectation-exp-cond-exp} to the left-hand side. 
Observe that 
\begin{align}
g_\theta(S_\lambda(u,s))
= & e^{\sk{\theta, \sqrt{1+\lambda}e^u(1+is)}}= 
g_{\theta\sqrt{1+\lambda}}(u,s).
\end{align}
Since $\E_{\mu^{W^\lambda}}[g_{\theta\sqrt{1+\lambda}}]=e^{\sk{\theta,\sqrt{1+\lambda}}}$ 
by Corollary \ref{co:expectation-generating-fnc}, the claim follows. 
\end{proof}

\smallskip
With these tools we can now prove the main result of this section. 

\smallskip\noindent
\begin{proof}[Proof of Theorem \ref{th:generating-exponential-martingale} --
Generating martingale]
The proof follows the same lines as the proof of Theorem \ref{th:first-mg}.
Recall that $\hat G^{(n)}$ is a function of $u^{(n)}$. Consequently, by 
Lemma \ref{le:measurability}, $\hat G^{(n)}$ and hence $M^{(n)}(\theta)$ are 
$\F_n$-measurable. To prove the martingale property, it suffices to show 
\begin{align}
\label{eq:expectation-mg2}
\E_{\mu^W_\infty}\Big[ M^{(n+1)}(\theta) \prod_{j\in V_n} e^{-\lambda_j{\pmb{\beta}_j}}\Big]
=\E_{\mu^W_\infty}\Big[ M^{(n)}(\theta) \prod_{j\in V_n} e^{-\lambda_j{\pmb{\beta}_j}} \Big]
\end{align}
for all $\lambda_i>-1$, $i\in V_n$.
Recall that by the construction in Section \ref{subse:infinite-graph}
the law of $\pmb{\beta}^{(n)}$ with respect to $\mu_\infty^W$ 
coincides with the law of $\beta^{V_n}$ with respect to $\mu_n^W$. Hence, 
we rewrite the claim \eqref{eq:expectation-mg2} in the form 
\begin{align}
\label{eq:expectation-mg3}
\E_{\mu^W_{n+1}}\Big[ \tilde M^{(n+1)}(\theta) \prod_{j\in V_n} 
e^{-\lambda_j\beta_j^{V_{n+1}}(u)} \Big]
=\E_{\mu^W_n}\Big[ \tilde M^{(n)}(\theta)
\prod_{j\in V_n} e^{-\lambda_j\beta_j^{V_n}(u)} \Big],
\end{align}
with the following variant of $M^{(n)}(\theta)$  
\begin{align}
\tilde M^{(n)}(\theta):\Omega_{V_n}\to\R, \quad 
\tilde M^{(n)}(\theta)=e^{\sk{\theta^{(n)}, e^u}-\frac12\sk{\theta^{(n)},\hat G^{V_n}(u)\theta^{(n)}}}.
\end{align}
Compare \eqref{eq:expectation-mg3} with the similar claim 
\eqref{eq:expectation-mg1}. Set $\lambda_i=0$ for 
$i\in\tilde V_{n+1}\setminus V_n$. Using Theorem \ref{thm:expectation-G}, 
claim \eqref{eq:expectation-mg3} is equivalent to 
\begin{align}
\label{eq:mg-property-general-case}
\cL^W_{n+1}(\lambda) e^{\sk{\theta^{(n+1)},\sqrt{1+\lambda}}}
=\cL^W_n(\lambda) e^{\sk{\theta^{(n)},\sqrt{1+\lambda}}}. 
\end{align}
By Lemma \ref{le:consistency}, one has $\cL^W_{n+1}(\lambda)=\cL^W_n(\lambda)$.
We calculate the remaining factors using the definition \eqref{def-delta-n}
of $\theta^{(i)}$, $i\in\{n,n+1\}$:
\begin{align}
\sk{\theta^{(n)},\sqrt{1+\lambda}}
= & \sum_{i\in V_n}\theta_i \sqrt{1+\lambda_i} + \theta_{\delta_n}^{(n)}
= \sum_{i\in V_n}\theta_i \sqrt{1+\lambda_i} + \sum_{j\in V_\infty\setminus V_n}\theta_j, 
\\
\sk{\theta^{(n+1)},\sqrt{1+\lambda}} 
= & \sum_{i\in V_{n+1}}\theta_i \sqrt{1+\lambda_i} 
+ \sum_{j\in V_\infty\setminus V_{n+1}}\theta_j 
= \sk{\theta^{(n)},\sqrt{1+\lambda}}, 
\end{align}
where in the last step we use $\lambda_{i}=0$ for $i\in\tilde V_{n+1}\setminus V_n$.
Thus, equality \eqref{eq:mg-property-general-case} holds and  
the martingale property is shown.
\end{proof}

\smallskip

\smallskip\noindent
\begin{proof}[Proof of Corollary \ref{cor:further-mgs} -- Hierarchy of martingales]
The random variable $M^{(n)}_{i_1,\ldots,i_m}$ is $\F_n$-measurable as a function of 
$u^{(n)}$ and $\hat G^{(n)}$. 
The martingale property for $M^{(n)}_{i_1,\ldots,i_m}$ is obtained
by expanding the corresponding property for $M^{(n)}(\theta)$ from
Theorem \ref{th:generating-exponential-martingale} around $\theta=0$,
as follows. We rewrite the martingale property for $M^{(n)}(\theta)$ 
in the following form: 
\begin{align}
\label{eq:mg-property-exp-theta}
\E_{\mu^W_\infty}\left[ M^{(n+1)}(\theta) 1_A({\pmb{\beta}}^{(n)}) \right]
=\E_{\mu^W_\infty}\left[ M^{(n)}(\theta) 1_A({\pmb{\beta}}^{(n)}) \right]
\end{align}
for any $A\in\B(\R^{V_n})$, $n\in\N$, $\theta\in(-\infty,0]^{(V_\infty)}$,
using the notation ${\pmb{\beta}}^{(n)}=({\pmb{\beta}}_i)_{i\in V_n}$ again. We take 
$m$ (left) partial derivatives of this equation with respect to $\theta$; note that 
the hypothesis $\theta_i\le 0$ and the fact that all moments of $\hat G^{(n)}$
are finite allow us to interchange expectations 
and partial derivatives. This yields
\begin{align}
\label{eq:mg-property-diff}
\E_{\mu^W_\infty}\left[ \partial_{\theta_{i_1}}\ldots\partial_{\theta_{i_m}} 
M^{(n+1)}(\theta) 1_A({\pmb{\beta}}^{(n)}) \right]
=\E_{\mu^W_\infty}\left[ \partial_{\theta_{i_1}}\ldots\partial_{\theta_{i_m}}
M^{(n)}(\theta) 1_A({\pmb{\beta}}^{(n)}) \right].
\end{align}
We use the well-known Isserlis-Wick-formula for $I\subseteq\{1,\ldots,m\}$
in the form 
\begin{align}
\label{eq:isserlis-wick}
\left. \left(\prod_{i\in I}\partial_{\theta_{j_i}}\right)
e^{-\frac12\sk{\theta^{(n)},\hat G^{(n)}\theta^{(n)}}} \right|_{\theta=0}
=\sum_{\I\in\mathcal{P}_2(I)} \prod_{\{k,l\}\in\I} (-\hat G^{(n)}_{k,l}).
\end{align}
The sum on the right-hand side is empty for 
sets $I$ with odd cardinality. Taking the iterated derivative of 
$M_n(\theta)$ as defined in \eqref{eq:def-M-n-theta}, using
the Leibniz rule and \eqref{eq:isserlis-wick}, we get
\begin{align}
& \left.
\partial_{\theta_{i_1}}\ldots\partial_{\theta_{i_m}}M^{(n)}(\theta)\right|_{\theta=0} 
\cr
& = \sum_{\substack{I\subseteq\{1,\ldots,m\}\\ |I| \text{ even}}}
e^{\sum_{k\in \{1,\ldots,m\}\setminus I} u^{(n)}_{i_k}}
\left. \left(\prod_{i\in I}\partial_{\theta_{j_i}}\right)
e^{-\frac12\sk{\theta^{(n)},\hat G^{(n)}\theta^{(n)}}} \right|_{\theta=0} 
= M_{i_1,\ldots,i_m}^{(n)}
.
\end{align}
Inserting this and the corresponding identity for $M^{(n+1)}(\theta)$ 
into \eqref{eq:mg-property-diff}
yields the martingale property for $M_{i_1,\ldots,i_m}^{(n)}$, $n\in\N$, in the form
\begin{align}
\label{eq:mg-property-mk}
\E_{\mu^W_\infty}\left[ M_{i_1,\ldots,i_m}^{(n+1)} 1_A({\pmb{\beta}}^{(n)}) \right]
=\E_{\mu^W_\infty}\left[ M_{i_1,\ldots,i_m}^{(n)} 1_A({\pmb{\beta}}^{(n)}) \right].
\end{align}
\end{proof}

\section{Proof of Letac's formula}
\label{se:letac}

\subsection{Special case $\phi =1$}
\label{se:proof-letac-special-case}
We consider first the simpler case
 $\phi=1$, i.e.\ 
$\phi_i=1$ for all $i\in V$. We will see later that the general
case follows by a scaling argument.

It is shown in Theorem 1 in \cite{sabot-tarres-zeng15} that the 
following  is a probability measure on $\R^V$:
\begin{align}
\nu(d\beta)=\nu^{W,1}(d\beta)
=1_{\{H_\beta>0\}}\left(\frac{2}{\pi}\right)^{\tfrac{|V|}{2}}
e^{-\sk{1,\beta}}\prod_{(i\sim j)\in E}e^{W_{ij}}\frac{1}{\sqrt{\det H_\beta}}\, d\beta.
\end{align}
Using the measure $\nu$, we obtain the relation 
\begin{align}
L:= & \left(\frac{2}{\pi}\right)^{\tfrac{|V|}{2}}\int_{\{b\in\R^V: H_b>0\}} 
\frac{e^{-\tfrac12\left(\sk{1,H_b 1}+\sk{\theta,H_b^{-1}\theta}\right)}
}{\sqrt{\det H_b}}\, db \cr 
= & \E_\nu\left[e^{\sk{1,\beta}
-\tfrac12\left(\sk{1,H_\beta1}+\sk{\theta,H_\beta^{-1}\theta}\right)}\right]
\prod_{(i\sim j)\in E}e^{-W_{ij}} 
=  \E_\nu\left[e^{-\tfrac12\sk{\theta,H_\beta^{-1}\theta}}\right],
\end{align}
where, in the last equality, we have used 
\begin{align}
\sk{1,\beta}-\tfrac12\sk{1,H_\beta1}-\sum_{(i\sim j)\in E}W_{ij}
= \tfrac12\sum_{i,j\in V}W_{ij}-\sum_{(i\sim j)\in E}W_{ij}=0.
\end{align}
The problem then reduces to evaluate $\E_{\nu}[e^{-\tfrac12\sk{\theta,H_\beta^{-1}\theta}}]$.
This is done in three steps.

\paragraph{Step 1.} 
Let  $\law_\nu(\beta)$ denote the 
law of $\beta=(\beta_i)_{i\in V}$ with respect to $\nu$. 
In  Corollary~2 of \cite{sabot-tarres-zeng15},  Sabot, Tarr{\`e}s, and Zeng
express $\law_\nu(\beta)$ using $\beta$ defined in analogy to 
\eqref{eq:beta-definition} and an additional independent
gamma distributed random variable, associated to a special vertex inside $V$. 

In contrast to this, here we consider again the enlarged vertex set
$\tilde V=V\cup\{\delta\}$ and the $H^{2|2}$
measure $\mu^W$, defined in \eqref{eq:def-mu-W}, on the  enlarged graph $(\tilde{V},\tilde{E})$.
We  may  assume the vertex $\delta\in\tilde V\setminus V$ is connected to a {\it single} vertex 
$\ell\in V$,
\begin{align}
\label{eq:def-E}
\tilde{E}=E\cup \{\ell\sim\delta\}, \quad 
E=\tilde E\setminus\{\ell\sim\delta\}.
\end{align}
We will prove below the following relation.

\begin{lemma}
\label{lemma-wc}
We have 
\begin{align}
\label{eq:equ-L}
L=\E_\nu\left[e^{-\tfrac12\sk{\theta,H_\beta^{-1}\theta}}\right]
=\lim_{W_{\ell\delta}\downarrow 0}\E_{\mu^W}\left[e^{-\tfrac12\sk{\theta,H_\beta^{-1}\theta}}\right], 
\end{align}
where $W_{\ell\delta}$ is the (positive) weight associated to the edge $\ell\sim \delta$.
\end{lemma}

\paragraph{Step 2.} 
To construct the analog of the additional gamma variable in  \cite{sabot-tarres-zeng15},
we  select now  as special vertex in $V$
the unique vertex $\ell$ connected to $\delta$.
 
Let us consider the reduced graph consisting of the vertex 
set $V^\circ=V\setminus\{\ell\}$ and edge set $E^\circ=E\setminus\{ (i\sim\ell): i\in V\}$.
In the same way, let $W^\circ\in\R^{V\times V}$ be the reduced weight matrix  given 
by $W^\circ_{ij}=W_{ij}$ for $i,j\in V$.
 
With respect to the smaller graph $\G=(V,E)$, the objects 
$V^\circ$, $E^\circ$, $\ell$, $V$, $E$, $W^\circ$, and $U_{V^\circ}$ play the same role as 
$V$, $E$, $\delta$, $\tilde V$, $\tilde E$, $W$, and $U_V$, with respect to the 
larger original graph $\tilde\G=(\tilde V,\tilde E)$.
In particular, we have the following analog of 
\eqref{eq:def-hat-G}:
\begin{equation}\label{eq:def-reduced-G}
\hat G^{V^\circ,W^\circ}\hspace{-0,2cm}:U_{V^\circ}\to\R^{V\times V}, \quad 
\hat G^{V^\circ,W^\circ}_{ij}(\tilde u)
= \left\{\begin{array}{ll}
e^{\tilde u_i}(A^{W^\circ}_{V^\circ V^\circ}(\tilde u)^{-1})_{ij}e^{\tilde u_j} & i,j\in V^\circ, \\
0 & i=\ell\text{ or }j=\ell. \\
\end{array}\right.
\end{equation}
Recall that  $H_{\beta(u)}^{-1}=\hat G^{V,W}_{VV}(u)$ by \eqref{eq:relation-H-beta-hat-G}.
To relate $U_{V}$ and $U_{V^{\circ}}$ we define the shift  
\begin{align}
\label{eq:def-tilde}
\mbox{}^\sim~:~U_V~\to~U_{V^\circ}, \quad
u\mapsto\tilde u=(\tilde u_i=u_i-u_\ell)_{i\in V}. 
\end{align}
Then we have the following
relation between $\hat G^{V,W}_{VV}(u)$ and $\hat G^{V^\circ,W^\circ}(\tilde u)$.

\begin{lemma}
\label{lemma-Grel} 
The matrices  $\hat G^{V,W}_{VV}(u)$ and $\hat G^{V^\circ,W^\circ}(\tilde u)$
satisfy the following relation
\begin{align}
\label{eq:relation-G-hat-G}
\hat G^{V,W}_{ij}(u) = \frac{e^{\tilde{u}_i+\tilde{u}_j}}{W_{\ell\delta} e^{-u_\ell}} + 
\hat G^{V^\circ,W^\circ}_{ij}(\tilde eu)
\quad\text{for all }i,j\in V. 
\end{align}
\end{lemma}
This relation is an analog to the second formula in 
Proposition 8 of \cite{sabot-zeng15}. The proof is given below.

\paragraph{Step 3.} 
Using \eqref{eq:relation-G-hat-G}, we get 
\begin{equation}
\sk{\theta,H_\beta^{-1}\theta}
= \sk{\theta,\hat G^{V,W}_{VV}(u)\theta}= 
 \sk{\theta,\hat G^{V^\circ,W^\circ}(\tilde u)\theta} 
+ \frac{e^{u_\ell}}{W_{\ell\delta}} \sk{\theta,e^{\tilde u}}^2.
\end{equation}
Inserting this in \eqref{eq:equ-L}, we obtain 
\begin{align}
\label{eq:L-umgeformt}
L= & \lim_{W_{\ell\delta}\downarrow 0}\E_{\mu^W}\left[e^{-\tfrac12\sk{\theta,H_\beta^{-1}\theta}}\right]
=\lim_{W_{\ell\delta}\downarrow 0}\E_{\mu^W}
\left[e^{-\tfrac12\sk{\theta,\hat G^{V^\circ,W^\circ}(\tilde u)\theta}
-\frac{e^{u_\ell}}{2W_{\ell\delta}}\sk{\theta,e^{\tilde u}}^2} \right] \cr
= & \lim_{W_{\ell\delta}\downarrow 0}\E_{\mu^W}
\left[e^{-\tfrac12\sk{\theta,\hat G^{V^\circ,W^\circ}(\tilde u)\theta}}
\E_{\mu^W}\left[\left. 
e^{-\frac{e^{u_\ell}}{2W_{\ell\delta}}\sk{\theta,e^{\tilde u}}^2} \right|\tilde u \right]\right].
\end{align}
In the following, we denote the $H^{2|2}$-measure $\mu^W$, defined in 
\eqref{eq:def-mu-W}, by $\mu^{W,\tilde\G}_\delta$, in order to stress the dependence on 
the graph $\tilde\G$ and the reference point $\delta$, which satisfies 
$u_\delta=0$. The conditional expectation is described in the following lemma.

\begin{lemma}
\label{lemma-cond-exp}
Let $\tilde{u}\in U_{V^{\circ}}$ be defined as in \eqref{eq:def-tilde}.
We have
\begin{align}
\label{eq:cond-expectation}
\E_{\mu^{W,\tilde\G}_\delta}\left[\left. 
e^{-\frac{e^{u_\ell}}{2W_{\ell\delta}}\sk{\theta,e^{\tilde u}}^2} \right|\tilde u \right]= 
e^{W_{\delta \ell} - \sqrt{W_{\delta\ell}^{2}+\sk{\theta,e^{\tilde u}}^2} }.
\end{align}
\end{lemma}
The proof uses independence of $\tilde{u}$ and $u_\ell$ with respect to 
$\mu^{W,\tilde\G}_\delta$. It is given below. 

Now, inserting \eqref{eq:cond-expectation} into \eqref{eq:L-umgeformt}, we obtain
\begin{equation}
L=  \lim_{W_{\ell\delta}\downarrow 0}\E_{\mu^{W,\tilde\G}_\delta}
\left[e^{-\tfrac12\sk{\theta,\hat G^{V^\circ,W^\circ}(\tilde u)\theta}}e^{W_{\delta \ell} - \sqrt{W_{\delta\ell}^{2}+\sk{\theta,e^{\tilde u}}^2} }
\right]=
\E_{\mu^{W,\tilde\G}_\delta}
\left[e^{-\tfrac12\sk{\theta,\hat G^{V^\circ,W^\circ}(\tilde u)\theta}-\sk{\theta,e^{\tilde u}}}
\right].
\end{equation}
The measure $\mu^{W,\tilde\G}_\delta$, on the  {\it bigger} weighted graph 
$(\tilde\G,W)$ with reference point $\delta$, is related to the 
measure   $\mu^{W^\circ,\G}_\ell$  on the 
{\it smaller} weighted graph $(\G,W^\circ)$ with reference point $\ell$ as follows. 
The  $\mu^{W,\tilde\G}_\delta$-law of 
$\tilde u=(\tilde u_i=u_i-u_\ell)_{i\in V}$,  with $u\in U_{V}$, equals the 
$\mu^{W^\circ,\G}_\ell$-law of $u=(u_i)_{i\in V}\in U_{V^{\circ}}$.
Hence, applying \eqref{eq:expectation-G-exp-beta} 
from Theorem \ref{thm:expectation-G}
with $-\theta$ and $\lambda=0$, we get
\begin{align}
L= & \E_{\mu^{W^\circ,\G}_\ell}
\left[e^{-\tfrac12\sk{\theta,\hat G^{V^\circ,W^\circ}(u)\theta}-\sk{\theta,e^u}}
\right]
=\cL^{W^\circ}(0)e^{\sk{-\theta,1}}=e^{-\sk{\theta,1}}.
\end{align}
This proves formula \eqref{eq:letac-xiaolin} in the special case $\phi=1$. 

Finally, we give the proof of Lemmas \ref{lemma-wc}-\ref{lemma-cond-exp}.

\medskip\noindent\begin{proof}[Proof of Lemma \ref{lemma-wc}]
By Proposition 1 of \cite{sabot-tarres-zeng15}, the Laplace transform of 
$\beta=(\beta_i)_{i\in V}$ with respect to $\nu$ is given by 
\begin{align}
\E_\nu\left[e^{-\sk{\lambda,\beta}}\right] 
= \prod_{(i\sim j)\in E} e^{W_{ij}(1-\sqrt{1+\lambda_i} \sqrt{1+\lambda_j} )}
\prod_{i\in V} \frac{1}{\sqrt{1+\lambda_i}}
\end{align}
for $\lambda\in(-1,\infty)^V$. Comparing with formula \eqref{eq:Laplace} from 
Theorem \ref{thm:expectation-G}, we find for these~$\lambda$
\begin{align}
\label{eq:relation-vu-mu} 
\E_\nu\left[e^{-\sk{\lambda,\beta}}\right] 
=\E_{\mu^W}\left[e^{-\sk{\lambda,\beta}}\right]e^{-W_{\ell\delta}(1-\sqrt{1+\lambda_\ell})}
\end{align}
with the $H^{2|2}$ measure $\mu^W$ defined in \eqref{eq:def-mu-W}. To see this, 
one may extend $\lambda$ by the additional value $\lambda_\delta=0$.

Both sides of \eqref{eq:relation-vu-mu} are complex analytic functions of 
$\lambda\in((-1,\infty)+i\R)^V$. The square root is understood as its principal 
branch, i.e.\ $\sqrt{r^2e^{2i\varphi}}=re^{i\varphi}$ for $r>0$, $-\pi<\varphi<\pi$. 
Although equation \eqref{eq:relation-vu-mu}
was derived for {\it real} $\lambda\in(-1,\infty)^V$ only, the identity theorem 
for holomorphic functions implies that it holds also for {\it complex}
$\lambda\in((-1,\infty)+i\R)^V$. 
The identity \eqref{eq:relation-vu-mu} holds for any value $W_{\ell\delta}>0$. 
Hence, for all $\lambda\in((-1,\infty)+i\R)^V$, one has 
\begin{align}
\label{eq:exp-nu-as-limit}
\E_\nu\left[e^{-\sk{\lambda,\beta}}\right] 
=\lim_{W_{\ell\delta}\downarrow 0}\E_{\mu^W}\left[e^{-\sk{\lambda,\beta}}\right],
\end{align}
where in the last limit $W_{ij}$ is kept fixed unless $\{i,j\}=\{\ell,\delta\}$. 
In particular, taking {\it imaginary} $\lambda\in (i\R)^V$, equation
\eqref{eq:exp-nu-as-limit} shows a pointwise convergence of Fourier transforms. 
We conclude that $\law_{\mu^W}(\beta)$ converges 
weakly to $\law_\nu(\beta)$ as $W_{\ell\delta}\downarrow 0$. 
Since $H_\beta$ is positive definite, $\sk{\theta,H_\beta^{-1}\theta}>0$. 
Note that $1_{\{H_\beta>0\}}\exp(-\tfrac12\sk{\theta,H_\beta^{-1}\theta})$ 
is a bounded function of $\beta\in\R^V$ and its set of discontinuities has 
$\nu$-measure $0$. Consequently, using weak convergence, the result follows.
\end{proof}

\medskip\noindent\begin{proof}[Proof of Lemma \ref{lemma-Grel}]
Remember that $\hat G^{V,W}_{VV}(u)= (H^W_{\beta^V(u)})^{-1}$.
By using the partition  $V=V^{\circ}\cup\{\ell \}$, we can write
\begin{align}
H^W_{\beta^V(u)}= \begin{pmatrix}
2\beta_{\ell} (u) & -W_{\ell V^{\circ}}\\
-W_{V^{\circ} \ell}& M\\
\end{pmatrix},
\end{align}
where $M:=(H^W_{\beta^V(u)})_{V^\circ V^\circ}$.
Since $\delta$ is not directly connected to any vertex in $V^\circ$, we have
$\beta^{W^\circ}(\tilde u)=\beta^{W}(u)_{V^\circ}$. Hence 
$(H^{W^\circ}_{\beta^{V^\circ}(\tilde u)})_{V^\circ V^\circ}=(H^W_{\beta^V(u)})_{V^\circ V^\circ}$, and we 
conclude 
\begin{align} 
\hat G^{V^\circ,W^\circ}_{V^\circ V^\circ}(\tilde u) = M^{-1}.
\end{align}
We can write $ (H^W_{\beta^V(u)})^{-1}$ using the following block-matrix inversion formula
\begin{align}
\begin{pmatrix}
A& C\\
D& B\\
\end{pmatrix}^{-1}= \begin{pmatrix}
b^{-1}&-b^{-1} CB^{-1} \\
-B^{-1}D b^{-1}& B^{-1}+ B^{-1}Db^{-1}CB^{-1}\\
\end{pmatrix}, \qquad \mbox{with}\ b= A-CB^{-1}D
\end{align}
which holds if $B$ and $b$ are invertible. 
In our case $b=2\beta_\ell (u)-W_{\ell V^{\circ}} M^{-1} W_{V^{\circ} \ell}$ is a scalar and 
\begin{align}
2\beta_{\ell} (u)= W_{\ell\delta } e^{-u_{\ell}} + \sum_{j\in V^{\circ}} W_{\ell j} e^{\tilde{u}_{j}}=
 W_{\ell\delta } e^{-u_{\ell}} +  W_{\ell V^{\circ}} e^{\tilde{u}}_{ V^{\circ}}
\end{align}
with $e^{\tilde{u}}_{ V^{\circ}}=(e^{\tilde u_i})_{i\in V^{\circ}}$. Hence
\begin{align}
(H^W_{\beta^V(u)})^{-1}= & \frac{1}{b} \begin{pmatrix}
1 &  W_{\ell V^{\circ}} M^{-1}  \\
 M^{-1} W_{V^{\circ} \ell}&  M^{-1} W_{V^{\circ} \ell}  W_{\ell V^{\circ}} M^{-1}\\
\end{pmatrix}+ \begin{pmatrix}
0&0\\
0& M^{-1}\\
\end{pmatrix}\cr 
= & \frac{1}{b} \begin{pmatrix} 
1 \\ 
M^{-1} W_{V^{\circ} \ell}\end{pmatrix}
\begin{pmatrix} 
1 \\ 
M^{-1} W_{V^{\circ} \ell}\end{pmatrix}^t
+\hat G^{V^\circ,W^\circ}(\tilde u).
\label{eq:relation-hat-G-5}
\end{align}
Now recall that,  by 
\eqref{eq:exp-u-in-terms-of-H-beta}, we have $H^W_{\beta(u)}e^u_V=W_{V\delta}$, for all $u\in U_{V}$.
Applying the same identity to the smaller graph $(V,E)$ with reference
point $\ell$ we obtain  $H^{W^{\circ}}_{\beta(\tilde{u})}e^{\tilde{u}}_{V^{\circ}}=W_{V^{\circ}\ell}$, for all 
$\tilde{u}\in U_{V^{\circ}}$. We obtain
\begin{align}
 M^{-1} W_{V^{\circ} \ell} = e^{\tilde{u}}_{V^{\circ}}, \quad
\begin{pmatrix} 
1 \\ 
M^{-1} W_{V^{\circ} \ell}\end{pmatrix}=e^{\tilde{u}}_V .
\end{align}
Furthermore, $b= W_{\ell\delta } e^{-u_{\ell}}$ and \eqref{eq:relation-hat-G-5} yields 
the claim written in matrix form:
\begin{align}
\label{eq:relation-G-hat-G-matrix}
\hat G^{V,W}_{VV}(u) = e^{\tilde u}_V \frac{1}{W_{\ell\delta}e^{-u_\ell}}(e^{\tilde u}_V)^t
+ \hat G^{V^\circ,W^\circ}(\tilde u).
\end{align}
\end{proof}

\medskip\noindent\begin{proof}[Proof of Lemma \ref{lemma-cond-exp}]
Let $\Gamma$ denote the graph consisting only of the two vertices  
$\ell$ and $\delta$ and the edge $\ell\sim\delta$ with weight $W_{\ell\delta}$
connecting them. Using Lemma A.1 of \cite{disertori-merkl-rolles2015a},
the laws of $u_\ell$ with respect to $\mu^{W,\tilde\G}_\delta$ and 
$\mu^{W_{\ell\delta},\Gamma}_\delta$ coincide and 
the gradient variables $\tilde u$ are independent of $u_\ell$ with respect to 
$\mu^{W,\tilde\G}_\delta$. 
Thus, abbreviating $C(\tilde u)=\sk{\theta,e^{\tilde u}}^2/(2W_{\ell\delta})$, we get
\begin{align}
\E_{\mu^{W,\tilde\G}_\delta}\left[\left. 
e^{-\frac{e^{u_\ell}}{2W_{\ell\delta}}\sk{\theta,e^{\tilde u}}^2}\right|\tilde u \right]
= & \left. \E_{\mu^{W,\tilde\G}_\delta}\left[e^{-ce^{u_\ell}}\right]\right|_{c=C(\tilde u)}
= \left. \E_{\mu^{W_{\ell\delta},\Gamma}_\delta}\left[e^{-ce^{u_\ell}}\right]\right|_{c=C(\tilde u)}.
\end{align}
In order to compute the last expectation, we exchange the role of $\delta$ and
$\ell$ using $\ell$ as new reference point. 
The $\mu^{W_{\ell\delta},\Gamma}_\delta$-law of $u_\ell-u_\delta$ has the 
Radon-Nikodym derivative $e^{u_\delta-u_\ell}$ with respect to the 
$\mu^{W_{\ell\delta},\Gamma}_\ell$-law of the same function $u_\ell-u_\delta$. 
To see this, note that $t=u_\ell-u_\delta$ has distribution
\begin{align}
\sqrt{\frac{W_{\ell\delta}}{2\pi}
e^{-W_{\ell\delta}(\cosh t-1)}e^{-\tfrac{t}{2}}\,dt}
\end{align}
with respect to $\mu^{W_{\ell\delta},\Gamma}_\delta$.
Hence, 
for any $c>0$, we obtain
\begin{align}
\E_{\mu^{W_{\ell\delta},\Gamma}_\delta}\left[e^{-ce^{u_\ell}}\right]
=& \E_{\mu^{W_{\ell\delta},\Gamma}_\delta}\left[e^{-ce^{u_\ell-u_\delta}}\right]\cr
=& \E_{\mu^{W_{\ell\delta},\Gamma}_\ell}\left[e^{u_\delta-u_\ell}e^{-ce^{u_\ell-u_\delta}}\right]
=\E_{\mu^{W_{\ell\delta},\Gamma}_\ell}\left[e^{u_\delta}e^{-ce^{-u_\delta}}\right].
\end{align}
Note that for the weighted graph $(\Gamma,W_{\ell\delta})$ with reference $\ell$, 
we have $\beta_\delta=\beta_\delta^{W_{\ell\delta}}=\frac12 W_{\ell\delta}e^{-u_\delta}$. 
Thus, abbreviating $\lambda_\delta=2c/W_{\ell\delta}$, we obtain from formula 
\eqref{eq:ex-eu} in Example \ref{ex:expectation-exp-u}
\begin{align}
\E_{\mu^{W_{\ell\delta},\Gamma}_\ell}\left[e^{u_\delta}e^{-ce^{-u_\delta}}\right]
=\E_{\mu^{W_{\ell\delta},\Gamma}_\ell}\left[e^{u_\delta}e^{-\lambda_\delta\beta_\delta}\right]
=\cL^{W_{\ell\delta}}(\lambda_\delta)\sqrt{1+\lambda_\delta}
= e^{W_{\ell\delta}(1-\sqrt{1+\lambda_\delta})};
\end{align}
in the last equality we used \eqref{eq:laplace-trafo-beta}
to calculate $\cL^{W_{\ell\delta}}(\lambda_\delta)$.
Summarizing, this shows
\begin{align}
\E_{\mu^{W,\tilde\G}_\delta}\left[\left. 
e^{-\frac{e^{u_\ell}}{2W_{\ell\delta}}\sk{\theta,e^{\tilde u}}^2}\right|\tilde u \right]=
e^{W_{\ell\delta}\left[1-\sqrt{1+2C(\tilde u)W_{\ell\delta}^{-1}}
\right]}= e^{ W_{\ell\delta}-\sqrt{W_{\ell\delta}^2+\sk{\theta,e^{\tilde u}}^2}  }.
\end{align}
\end{proof}

\subsection{General case}
We deduce the general case of \eqref{eq:letac-xiaolin} from the
special case $\phi=1$ using a scaling argument. In this part of the proof,
we write $H_b^W$ rather than $H_b$ because we are working with varying weights $W$.
Let $\phi,\theta\in(0,\infty)^V$.
We consider the change of variables $b_i'=\phi_i^2 b_i$ for all $i\in V$ and
the rescaled weights $W_{ij}'=\phi_iW_{ij}\phi_j$.  Denoting by 
$\diag\phi\in\R^{V\times V}$ the diagonal matrix with diagonal entries 
$\phi_i$, $i\in V$, we have for $i,j\in V$ and $b\in\R^V$
\begin{align}
(\diag\phi\, H_b^W\diag\phi)_{ij}
=2\phi_ib_i\phi_j\delta_{ij}-\phi_iW_{ij}\phi_j
=2b_i'\delta_{ij}-W_{ij}'
=(H_{b'}^{W'})_{ij}.
\end{align}
Thus, $\diag\phi\, H_b^W\diag\phi=H_{b'}^{W'}$. From this, we deduce 
\begin{align}
(H_b^W)^{-1}=\diag\phi\, (H_{b'}^{W'})^{-1}\diag\phi
\quad\text{ and }\quad
\frac{1}{\sqrt{\det H_b^W}}=\frac{\prod_{i\in V}\phi_i}{\sqrt{\det H_{b'}^{W'}}}.
\end{align}
Furthermore, $H_b^W>0$ if and only if $H_{b'}^{W'}>0$. Changing variables 
from $b$ to $b'$ we get the Jacobi determinant $|db/db'|=(\prod_{i\in V}\phi_i)^{-2}$. 
Abbreviating $\theta\phi=(\theta_i\phi_i)_{i\in V}$, we conclude
\begin{align}
& \int_{\{b\in\R^V: H_b^W>0\}} 
\frac{e^{-\tfrac12\left(\sk{\phi,H_b^W\phi}+\sk{\theta,(H_b^W)^{-1}\theta}\right)}
}{\sqrt{\det H_b^W}}\, db \cr 
= & \int_{\{b'\in\R^V: H_{b'}^{W'}>0\}} 
\frac{e^{-\tfrac12\left(\sk{1,H_{b'}^{W'}1}+\sk{\theta\phi,(H_{b'}^{W'})^{-1}\theta\phi}\right)}
}{\sqrt{\det H_{b'}^{W'}}\prod_{i\in V}\phi_i}\, db' 
= \left(\frac{\pi}{2}\right)^{\tfrac{|V|}{2}}
\frac{e^{-\sk{1,\phi\theta}}}{\prod_{i\in V}\phi_i},
\end{align}
where we used \eqref{eq:letac-xiaolin} for the special case $\phi=1$ treated in 
Section \ref{se:proof-letac-special-case} above.
Since $\sk{1,\phi\theta}=\sk{\phi,\theta}$ the claim \eqref{eq:letac-xiaolin} follows.

\vspace{1cm}\noindent
\textbf{Acknowledgments:} We would like to thank Christophe Sabot and 
Xiaolin Zeng for explaining their results and Letac's formula
to us in different occasions.


\begin{thebibliography}{DMR15}

\bibitem[DMR15]{disertori-merkl-rolles2015a}
M.~Disertori, F.~Merkl, and S.W.W. Rolles.
\newblock A comparison of the nonlinear sigma model with general pinning and
  pinning at one point.
\newblock Preprint, arXiv:1506.01852, 2015.

\bibitem[DS10]{disertori-spencer2010}
M.~Disertori and T.~Spencer.
\newblock Anderson localization for a supersymmetric sigma model.
\newblock {\em Comm. Math. Phys.}, 300(3):659--671, 2010.

\bibitem[DSZ10]{disertori-spencer-zirnbauer2010}
M.~Disertori, T.~Spencer, and M.R. Zirnbauer.
\newblock Quasi-diffusion in a 3{D} supersymmetric hyperbolic sigma model.
\newblock {\em Comm. Math. Phys.}, 300(2):435--486, 2010.

\bibitem[ST15]{sabot-tarres2012}
C.~Sabot and P.~Tarr{\`e}s.
\newblock Edge-reinforced random walk, vertex-reinforced jump process and the
  supersymmetric hyperbolic sigma model.
\newblock {\em JEMS}, 17(9):2353--2378, 2015.

\bibitem[STZ15]{sabot-tarres-zeng15}
C.~Sabot, P.~Tarr{\`e}s, and X.~Zeng.
\newblock A new exponential family related to the vertex reinforced jump
  process.
\newblock Preprint arXiv:1507.04660, 2015.

\bibitem[SZ15]{sabot-zeng15}
C.~Sabot and X.~Zeng.
\newblock A random {Schr{\"o}dinger} operator associated with the vertex
  reinforced jump process and the edge reinforced random walk.
\newblock Preprint arXiv:1507.07944, 2015.

\bibitem[Zir91]{zirnbauer-91}
M.R. Zirnbauer.
\newblock Fourier analysis on a hyperbolic supermanifold with constant
  curvature.
\newblock {\em Comm. Math. Phys.}, 141(3):503--522, 1991.

\end{thebibliography}
\end{document}